\documentclass{amsart}
\usepackage{latexsym,amsfonts,amsmath,amssymb}
\usepackage{diagrams}
\usepackage{verbatim}
\diagramstyle[tight,centredisplay%
]%
\begin{document}
%------------------Some math macros------------------
\def\<#1>{\langle#1\rangle}
\newcommand{\ZFC}{{\rm ZFC}}
\newcommand{\GBC}{{\rm GBC}}
\newcommand{\power}{\mathcal P}
\newcommand{\restrict}{\upharpoonright}
\newcommand{\B}{{\mathbb B}}
\renewcommand{\P}{{\mathbb P}}
\newcommand{\Q}{{\mathbb Q}}
\newcommand{\R}{{\mathbb R}}
\newcommand{\Vbar}{{\overline{V}}}
\newcommand{\Qdot}{{\dot\Q}}
\newcommand{\Rdot}{{\dot\R}}
\newcommand{\Ptail}{{\dot{\P}_\tail}}
\newcommand{\Gtail}{{G_\tail}}
\newcommand{\smalllt}{\mathrel{\mathchoice{\raise2pt\hbox{$\scriptstyle<$}}{\raise1pt\hbox{$\scriptstyle<$}}{\raise0pt\hbox{$\scriptscriptstyle<$}}{\scriptscriptstyle<}}}
\newcommand{\smallleq}{\mathrel{\mathchoice{\raise2pt\hbox{$\scriptstyle\leq$}}{\raise1pt\hbox{$\scriptstyle\leq$}}{\raise1pt\hbox{$\scriptscriptstyle\leq$}}{\scriptscriptstyle\leq}}}
\newcommand{\ltkappa}{{{\smalllt}\kappa}}
\newcommand{\leqkappa}{{{\smallleq}\kappa}}
\newcommand{\leqgamma}{{{\smallleq}\gamma}}
\newcommand{\ltgamma}{{{\smalllt}\gamma}}
\newcommand{\leqtheta}{{{\smallleq}\theta}}
\newcommand{\lttheta}{{{\smalllt}\theta}}
\newcommand{\leqdelta}{{{\smallleq}\delta}}
\newcommand{\ltdelta}{{{\smalllt}\delta}}
\newcommand{\plus}{{+}}
\newcommand{\plusplus}{{{+}{+}}}
\newcommand{\plusplusplus}{{{+}{+}{+}}}
\newcommand{\satisfies}{\models}
\newcommand{\forces}{\Vdash}
\newcommand{\CH}{{\rm CH}}
\newcommand{\LC}{{\rm LC}}
\newcommand{\GCH}{{\rm GCH}}
\newcommand{\SCH}{{\rm SCH}}
\newcommand{\PFA}{{\rm PFA}}
\newcommand{\ORD}{\mathop{{\rm ORD}}}
\newcommand{\Ord}{\mathop{{\rm Ord}}}
\newcommand{\HOD}{\mathop{{\rm HOD}}}
\newcommand{\CARD}{\mathop{{\rm CARD}}}
\newcommand{\df}{\it}
\newcommand{\of}{\subseteq}
\newcommand{\boolval}[1]{\mathopen{\lbrack\!\lbrack}\,#1\,\mathclose{\rbrack\!\rbrack}}
\newcommand{\ran}{\mathop{\rm ran}}
\newcommand{\Ult}{\mathop{\rm Ult}}
\newcommand{\val}{\mathop{\rm val}\nolimits}
\newcommand{\one}{\mathop{1\hskip-2.5pt {\rm l}}}
\newcommand{\image}{\mathbin{\hbox{\tt\char'42}}}
\newcommand{\st}{\mid}
\newcommand{\set}[1]{\{\,{#1}\,\}}
\newcommand{\cp}{\mathop{\rm cp}}
\newcommand{\union}{\cup}
\newcommand{\squnion}{\sqcup}
\newcommand{\Union}{\bigcup}
\newcommand{\intersect}{\cap}
\newcommand{\Intersect}{\bigcap}
\newcommand{\elesub}{\prec}
\renewcommand{\th}{{\hbox{\scriptsize th}}}
\newcommand{\her}[1]{H_{{#1}^+}}
\newcommand{\from}{\mathbin{\vbox{\baselineskip=2pt\lineskiplimit=0pt
                         \hbox{.}\hbox{.}\hbox{.}}}}
\newcommand{\Add}{\mathop{\rm Add}}
\newcommand{\tail}{\text{tail}}
\newtheorem{theorem}{Theorem}
\newtheorem{lemma}[theorem]{Lemma}
\newtheorem{corollary}[theorem]{Corollary}
\theoremstyle{definition}
\newtheorem{definition}[theorem]{Definition}
%---------------------------------------------------
%
\author{Arthur W. Apter}
\address{A. W. Apter, Mathematics, The Graduate Center of The City University of New York, 365 Fifth Avenue, New York, NY 10016
\& Department of Mathematics, Baruch College of CUNY,
One Bernard Baruch Way, New York, NY 10010}
\email{awapter@alum.mit.edu,
http://faculty.baruch.cuny.edu/aapter}
\author{Victoria Gitman}
\address{V. Gitman, Mathematics, New York City College of Technology, 300 Jay Street, Brooklyn, NY 11201}
\email{vgitman@nylogic.org,
http://websupport1.citytech.cuny.edu/faculty/vgitman}
\author{Joel David Hamkins}
\address{J. D. Hamkins, Department of Philosophy, New York University, 5 Washington Place,
New York, NY 10003, \& Mathematics, The Graduate Center of
The City University of New York, 365 Fifth Avenue, New
York, NY 10016
 \& Mathematics, The College of Staten Island of CUNY, Staten Island, NY 10314}
\email{jhamkins@gc.cuny.edu, http://jdh.hamkins.org}
%\today  % September 4th, 2007
\thanks{The research of each of the authors has been supported
in part by research grants from the CUNY Research
Foundation. The third author's research has been
additionally supported by research grants from the National
Science Foundation and from the Simons
Foundation.}
\date{June 25, 2010 (revised October 31, 2011)}
%\date{\today}

\begin{abstract}
We construct a variety of inner models exhibiting features
usually obtained by forcing over universes with large
cardinals. For example, if there is a supercompact
cardinal, then there is an inner model with a Laver
indestructible supercompact cardinal. If there is a
supercompact cardinal, then there is an inner model with a
supercompact cardinal $\kappa$ for which
$2^\kappa=\kappa^\plus$, another for which
$2^\kappa=\kappa^\plusplus$ and another in which the least
strongly compact cardinal is supercompact. If there is a
strongly compact cardinal, then there is an inner model
with a strongly compact cardinal, for which the measurable
cardinals are bounded below it and another inner model $W$
with a strongly compact cardinal $\kappa$, such that
$H_{\kappa^\plus}^V\of\HOD^W$. Similar facts hold for
supercompact, measurable and strongly Ramsey cardinals. If
a cardinal is supercompact up to a weakly iterable
cardinal, then there is an inner model of the Proper
Forcing Axiom and another inner model with a supercompact
cardinal in which $\GCH+V=\HOD$ holds. Under the same
hypothesis, there is an inner model with level by level
equivalence between strong compactness and
supercompactness,
%another in which the least measurable cardinal is strongly compact
and indeed, another in which there is level by level
inequivalence between strong compactness and
supercompactness. If a cardinal is strongly compact up to a
weakly iterable cardinal, then there is an inner model in
which the least measurable cardinal is strongly compact. If
there is a weakly iterable limit $\delta$ of
$\ltdelta$-supercompact cardinals, then there is an inner
model with a proper class of Laver-indestructible
supercompact cardinals. We describe three general proof
methods, which can be used to prove many similar results.
\end{abstract}

\title[Inner models with large cardinal properties]{Inner
models with large cardinal features usually obtained by
forcing} \maketitle

\section{Introduction}

The theme of this article is to investigate the extent to
which several set-theoretic properties obtainable by
forcing over universes with large cardinals must also
already be found in an inner model. We find this
interesting in the case of supercompact and other large
cardinals that seem to be beyond the current reach of the
fine-structural inner model program. For example, one
reason we know that the \GCH\ is relatively consistent with
many large cardinals, especially the smaller large
cardinals, is that the fine-structural inner models that
have been constructed for these large cardinals satisfy the
\GCH; another reason is that the canonical forcing of the
\GCH\ preserves all the standard large cardinals. In the
case of supercompact and other very large large cardinals,
we currently lack such fine-structural inner models and
therefore have relied on the forcing argument alone when
showing relative consistency with the \GCH. It seems quite
natural to inquire, without insisting on fine structure,
whether these cardinals nevertheless have an inner model
with the \GCH.

\newtheorem{testquestion}[theorem]{Test Question}
\begin{testquestion}\label{TestQuestion.Supercompact+GCH}
If there is a supercompact cardinal, then must there be an
inner model with a supercompact cardinal in which the \GCH\
also holds? %(see theorem \ref{Theorem.InnerModelSC+GCH})
\end{testquestion}

\begin{testquestion}\label{TestQuestion.Supercompact+2^kappa=kappa+}
If there is a supercompact cardinal, then must there be an
inner model with a supercompact cardinal $\kappa$ such that
$2^\kappa=\kappa^\plus$? %(see theorem \ref{Theorem.InnerModelSC+2^kappa=kappa+})
\end{testquestion}

\begin{testquestion}\label{TestQuestion.Supercompact+2^kappa>kappa+}
If there is a supercompact cardinal, then must there be an
inner model with a supercompact cardinal $\kappa$ such that
$2^\kappa>\kappa^\plus$? %(see theorem \ref{Theorem.InnerModelSC+2^kappa=kappa+})
\end{testquestion}

These questions are addressed by our Theorems
\ref{Theorem.InnerModelSC+2^kappa=kappa+} and
\ref{Theorem.InnerModelSC+GCH}. We regard these
test questions and the others we are about to introduce as
stand-ins for their numerous variations, asking of a
particular set-theoretic assertion known to be forceable
over a universe with large cardinals, whether it must hold
already in an inner model whenever such large cardinals
exist. The questions therefore concern what we describe as
the internal consistency strength of the relevant
assertions, a concept we presently explain. Following ideas
of Sy Friedman
\cite{Friedman2006:InternalConsistencyAndIMH}, let us say
that an assertion $\varphi$ is {\df internally consistent}
if it holds in an inner model, that is, if there is a
transitive class model of \ZFC, containing all the
ordinals, in which $\varphi$ is true. In this general form,
an assertion of internal consistency is a second-order
assertion, expressible in \GBC\ set theory
(as are our test questions); nevertheless,
it turns out that many
interesting affirmative instances of internal consistency
are expressible in the first-order language of set theory,
when the relevant inner model is a definable class, and as
a result much of the analysis of internal consistency can
be carried out in first-order \ZFC. One may measure what we
refer to as the {\df internal consistency strength} of an
assertion $\varphi$ by the hypothesis necessary to prove
that $\varphi$ holds in an inner model. Specifically, we
say that the internal consistency strength of $\varphi$ is
bounded above by a large cardinal or other hypothesis
$\psi$, if we can prove from $\ZFC+\psi$ that there is an
inner model of $\varphi$; in other words, if we can argue
from the truth of $\psi$ to the existence of an inner model
of $\varphi$. Two statements are {\df
internally-equiconsistent} if each of them proves the
existence of an inner model of the other. It follows that
the internal consistency strength of an assertion is at
least as great as the ordinary consistency strength of that
assertion, and the interesting phenomenon here is that
internal consistency strength can sometimes exceed ordinary
consistency strength. For example, although the hypothesis
$\varphi$ asserting ``there is a measurable cardinal and
\CH\ fails'' is equiconsistent with a measurable cardinal,
because it is easily forced over any model with a
measurable cardinal, nevertheless the internal consistency
strength of $\varphi$, assuming consistency, is strictly
larger than a measurable cardinal, because there are models
having a measurable cardinal in which there is no inner
model satisfying $\varphi$. For example, in the canonical
model $L[\mu]$ for a single measurable cardinal, every
inner model with a measurable cardinal contains an iterate
of $L[\mu]$ and therefore agrees that \CH\ holds. So one
needs more than just a measurable cardinal in order to
ensure that there is an inner model with a measurable
cardinal in which \CH\ fails.

With this sense of internal consistency strength, the
reader may observe that our test questions exactly inquire
about the internal consistency strength of their
conclusions. For instance, Test Questions
\ref{TestQuestion.Supercompact+GCH},
\ref{TestQuestion.Supercompact+2^kappa=kappa+} and
\ref{TestQuestion.Supercompact+2^kappa>kappa+} inquire
whether the internal consistency strength of a supercompact
cardinal plus the corresponding amount of the \GCH\ or its
negation is bounded above by and hence
internally-equiconsistent with the existence of a
supercompact cardinal.

In several of our answers, the inner models we provide will
also exhibit additional nice features; for example, in some
cases we shall produce for every cardinal $\theta$ an inner
model $W$ satisfying the desired assertion, but also having
$W^\theta\of W$. These answers therefore provide an
especially strong form of internal consistency, and it
would be interesting to investigate the extent to which the
strong internal consistency strength of an assertion can
exceed its internal consistency strength, which as we have
mentioned is already known sometimes to exceed its ordinary
consistency strength.

Let us continue with a few more test questions that we
shall use to frame our later discussion. Forcing, of
course, can also achieve large cardinal properties that we
do not expect to hold in the fine-structural inner models.
For example, Laver \cite{Laver78} famously proved that
after his forcing preparation, any supercompact cardinal
$\kappa$ is made {\df \emph{(}Laver\emph{)}
indestructible}, meaning that it remains supercompact after
any further $\ltkappa$-directed closed forcing. In
contrast, large cardinals are typically destructible over
their fine-structural inner models (for example, see
\cite[Theorem 1.1]{Hamkins94:FragileMeasurability}, and
observe that the argument generalizes to many of the other
fine-structural inner models; the crucial property needed
is that the embeddings of the forcing extensions of the
fine-structural model should lift ground model embeddings).
Nevertheless, giving up the fine-structure, we may still
ask for indestructibility in an inner model.

\begin{testquestion}\label{TestQuestion.IndestructibleSC}
If there is a supercompact cardinal, then must there be an
inner model with an indestructible supercompact cardinal?
\end{testquestion}

We answer this question in Theorem
\ref{Theorem.InnerModelIndesctructibleSC}. For another
example, recall that Baumgartner \cite{baumgartner:pfa}
proved that if $\kappa$ is a supercompact cardinal, then
there is a forcing extension satisfying the Proper Forcing Axiom
(\PFA). We inquire whether there must in fact be an inner
model satisfying the \PFA:

\begin{testquestion}\label{TestQuestion.PFA}
If there is a supercompact cardinal, then must there be an
inner model satisfying the Proper Forcing Axiom?
\end{testquestion}

This question is addressed by Theorem
\ref{Theorem.InnerModelPFA}, using a stronger hypothesis.
Next, we inquire the extent to which there must be inner
models $W$ having a very rich $\HOD^W$. With class forcing,
one can easily force $V=\HOD$ while preserving all of the
most well-known large cardinal notions, and of course, one
finds $V=\HOD$ in the canonical inner models of large
cardinals. Must there also be such inner models for the
very large large cardinals?

\begin{testquestion}\label{TestQuestion.SC+V=HOD}
If there is a supercompact cardinal, then must there be an
inner model with a supercompact cardinal satisfying
$V=\HOD$?
\end{testquestion}
Since one may easily force to make any particular set $A$
definable in a forcing extension by forcing that preserves
all the usual large cardinals, another version of this
question inquires:

\begin{testquestion}\label{TestQuestion.SC+HOD^W}
If there is a supercompact cardinal, then for every set
$A$, must there be an inner model $W$ with a supercompact
cardinal such that $A\in\HOD^W$?
\end{testquestion}

These questions are addressed by our Theorems
\ref{Theorem.HOD^W} and \ref{Theorem.InnerModelSC+GCH}. One
may similarly inquire, if there is a measurable cardinal,
then does every set $A$ have an inner model with a
measurable cardinal in which $A\in\HOD^W$? What of other
large cardinal notions? What if one restricts to $A\in
H_{\kappa^\plus}$? There is an enormous family of such
questions surrounding the $\HOD$s of inner models.
Furthermore, apart from large cardinals, for which sets $A$
is there an inner model $W$ with $A\in\HOD^W$? There are
numerous variants of this question.

More generally, whenever a feature is provably forceable in
the presence of a certain large cardinal, then we ask: is
there already an inner model with that feature? How robust
can these inner models be?

Before continuing, we fix some terminology. Suppose
$\kappa$ is a regular cardinal. A forcing notion is {\df
$\ltkappa$-directed closed} when any directed subset of it
of size less than $\kappa$ has a lower bound. (This is what
Laver in \cite{Laver78} refers to as {\df $\kappa$-directed
closed}.) A forcing notion is {\df $\leqkappa$-closed} if
any decreasing chain of length less than or equal to
$\kappa$ has a lower bound. A forcing notion is {\df
$\leqkappa$-strategically closed} if in the game of length
$\kappa + 1$ in which two players alternately select
conditions from it to construct a descending $(\kappa +
1)$-sequence, with the second player playing at limit
stages, the second player has a strategy that allows her
always to continue playing. A forcing notion is {\df
${<}\kappa$-strategically closed} if in the game of length
$\kappa$ in which two players alternately select conditions
from it to construct a descending $\kappa$-sequence, with
the second player playing at limit stages, the second
player has a strategy that allows her always to continue
playing. If a poset $\P$ is $\leqkappa$-closed, then it is
also $\leqkappa$-strategically closed. If $\lambda$ is an
ordinal, then ${\rm Add}(\kappa, \lambda)$ is the standard
poset for adding $\lambda$ many Cohen subsets to $\kappa$.
A Boolean algebra $\B$ is {\df $(\lambda, 2)$-distributive}
if the distributive law:
$\bigwedge_{\alpha<\lambda}u_{\alpha,0}\vee
u_{\alpha,1}=\bigvee_{f\in
2^\lambda}\bigwedge_{\alpha<\lambda} u_{\alpha,f(\alpha)}$
holds. Equivalently, a Boolean algebra is
$(\lambda,2)$-distributive if every $f : \lambda \to 2$ in
the generic extension by $\B$ is in the ground model. The
theory $\ZFC^-$ consists of the standard \ZFC\ axioms
without the powerset axiom and with the replacement scheme
replaced by the collection scheme (see
\cite{zfcminus:gitmanhamkinsjohnstone} for the significance
of choosing collection over replacement). A transitive set
$M \models {\rm ZFC}^-$ is a {\df $\kappa$-model} if $|M| =
\kappa$, $\kappa \in M$ and $M^{< \kappa} \subseteq M$. An
elementary embedding $j:M\to N$ is said to {\it lift} to
another elementary embedding $j^*:M^*\to N^*$, where $M\of
M^*$ and $N\of N^*$, if the two embeddings agree on the
smaller domain, i.e. $j^*\restrict M = j$. An elementary
embedding $j : M \to N$ having critical point $\kappa$ is
{\df $\kappa$-powerset preserving} if $M$ and $N$ have the
same subsets of $\kappa$. A cardinal $\kappa$ is {\df
strongly Ramsey} if every $A \subseteq \kappa$ is contained
in a $\kappa$-model $M$ for which there exists a
$\kappa$-powerset preserving elementary embedding $j : M
\to N$.

\section{Three Proof Methods}\label{sec:threeproofs}

In order best to introduce our methods, which we view as
the main contribution of this article, we shall begin with
Test Question \ref{TestQuestion.IndestructibleSC}, which is
answered by Theorem
\ref{Theorem.InnerModelIndesctructibleSC} below. We shall
give three different arguments with this conclusion, using
different proof methods (our third method will prove a
slightly weaker result, because it requires a slightly
stronger hypothesis). These methods are robust enough
directly to answer many variants of the test questions. In
Sections \ref{Section.ThirdMethod} and
\ref{Section.FurtherApplications}, we describe how some
further modifications of the methods enable them to prove
additional related results.

\begin{theorem}\label{Theorem.InnerModelIndesctructibleSC}
If there is a supercompact cardinal, then there is an inner
model with an indestructible supercompact cardinal.
\end{theorem}

The first proof makes use of an observation of Hamkins and
Seabold involving Boolean ultrapowers (see
\cite{HamkinsSeabold:BooleanUltrapowers}), which is
essentially encapsulated in Theorems
\ref{Theorem.FriendlyStrC} and
\ref{Theorem.FriendlyStrCemb}.
\begin{definition}
A  forcing notion $\P$ is {\df $\ltkappa$-friendly} if for
every $\gamma<\kappa$, there is a condition $p\in\P$ below
which the restricted forcing $\P\restrict p$ adds no
subsets to $\gamma$.
\end{definition}

\begin{theorem}[Hamkins, Seabold \cite{HamkinsSeabold:BooleanUltrapowers}]\label{Theorem.FriendlyStrC}
If $\kappa$ is a strongly compact cardinal and $\P$ is a
$\ltkappa$-friendly notion of forcing, then there is an
inner model $W$ satisfying every sentence forced by $\P$
over $V$.
\end{theorem}

\begin{proof}
The proof uses Boolean ultrapowers (see
\cite{HamkinsSeabold:BooleanUltrapowers} for a full
account). To make this paper self-contained, we shall
review the method. Suppose that $\B$ is the complete
Boolean algebra corresponding to the forcing notion $\P$.
Let $V^\B$ be the usual class of $\B$-names, endowed as a
Boolean-valued structure by the usual recursive definition
of the Boolean values $\boolval{\varphi}$ for every
assertion $\varphi$ in the forcing language. Now suppose
that $U\of\B$ is an ultrafilter, not necessarily generic in
any sense, and define the equivalence relation
$\sigma=_U\tau\iff \boolval{\sigma=\tau}\in U$. When $U$ is
not $V$-generic, this relation is not the same as
$\val(\sigma,U)=\val(\tau,U)$. Nevertheless, the relation
$\sigma\in_U\tau\iff\boolval{\sigma\in\tau}\in U$ is
well-defined with respect to $=_U$, and we may form the
quotient structure $V^\B/U$ as the collection of (Scott's
trick reduced) equivalence classes $[\tau]_U$. The relation
$\in_U$ is set-like, because whenever $\sigma\in_U\tau$,
then $\sigma$ is $=_U$ equivalent to a mixture of the names
in the domain of $\tau$, and there are only set many such
mixtures. One can easily establish \L os' theorem that
$V^\B/U\satisfies\varphi[[\tau]_U]\iff
\boolval{\varphi(\tau)}\in U$. In particular, any statement
$\varphi$ that is forced by $\one$ will be true in
$V^\B/U$. Thus, since $U$ is in $V$, we have produced in
$V$ a class model $V^\B/U$ satisfying the desired theory;
but there is no reason so far to suppose that this model is
well-founded.

In order to find an ultrafilter $U$ for which $V^\B/U$ is
well-founded, we shall make use of our assumption that $\P$
and hence also $\B$ is $\ltkappa$-friendly for a strongly
compact cardinal $\kappa$. Just as with classical powerset
ultrapowers, the structure $V^\B/U$ is well-founded if and
only if $U$ is countably complete (see
\cite{HamkinsSeabold:BooleanUltrapowers}). Next, consider
any $\theta\geq |\B|$ and let $j:V\to M$ be a
$\theta$-strong compactness embedding, so that
$j\image\B\of s\in M$ for some $s\in M$ with
$|s|^M<j(\kappa)$. Since $j(\B)$ is
${\smalllt}j(\kappa)$-friendly, there is a condition $p\in
j(\B)$ such that $j(\B)\restrict p$ adds no new subsets to
$\lambda=|s|^M$. Thus, $j(\B)\restrict p$ is
$(\lambda,2)$-distributive in $M$. Applying this, it
follows in $M$ that $$p=p\wedge 1=p\wedge\bigwedge_{b\in
s}(b\vee\neg b)=\bigwedge_{b\in s}(p\wedge
b)\vee(p\wedge\neg b)=\bigvee_{f\in2^s}\bigwedge_{b\in
s}(p\wedge(\neg)^{f(b)} b),$$ where $(\neg)^0b=b$ and
$(\neg)^1b=\neg b$, and where we use distributivity to
deduce the final equality. Since $p$ is not $0$, it follows
that there must be some $f$ with $q=\bigwedge_{b\in
s}p\wedge(\neg)^{f(b)}b\neq 0$. Note that $f(b)$ and
$f(\neg b)$ must have opposite values. Now we use $q$ as a
seed to define the ultrafilter $U=\set{a\in \B\st q\leq
j(a)}$, which is the same as $\set{a\in\B\st f(j(a))=0}$.
This is easily seen to be a $\kappa$-complete filter using
the fact that $\cp(j)=\kappa$ (just as in the powerset
ultrafilter cases known classically). It is an ultrafilter
precisely because $s$ covers $j\image\B$, so either
$f(j(a))=0$ or $f(\neg j(a))=0$, and so either $a\in U$ or
$\neg a\in U$, as desired. In summary, using this
ultrafilter $U$, the structure $V^\B/U$ is a well-founded
set-like model of the desired theory. The corresponding
Mostowski collapse is the desired inner model $W$.
\end{proof}

The metamathematical reader will observe that Theorem
\ref{Theorem.FriendlyStrC} is more properly described as a
theorem scheme, since we defined a certain inner model,
using $\P$ as a parameter, and then proved of each sentence
forceable by $\P$ over $V$, that this sentence also holds
in the inner model. By Tarski's theorem on the
non-definability of truth, it does not seem possible to
state the conclusion of Theorem \ref{Theorem.FriendlyStrC}
in a single first order statement. Several similar theorems
in this article will also be theorem schemes.

The following account of the Boolean ultrapower may be
somewhat more illuminating.

\begin{theorem}[\cite{HamkinsSeabold:BooleanUltrapowers}]\label{Theorem.FriendlyStrCemb}
If $\kappa$ is strongly compact and $\P$ is
$\ltkappa$-friendly, then there is an elementary embedding
$j:V\to\Vbar$ into an inner model $\Vbar$ and a
$\Vbar$-generic filter $G\of j(\P)$ with $G\in V$. In
particular, $W=\Vbar[G]$ fulfills Theorem
\ref{Theorem.FriendlyStrC}.
\end{theorem}

\begin{proof}
This is actually what is going on in the Boolean quotient.
We may define the canonical predicate for the ground model
of $V^\B$ by $\boolval{\tau\in\check V}=\bigvee_{x\in
V}\boolval{\tau=\check x}$, and let $\Vbar=\{[\tau]_U\mid
\boolval{\tau\in\check V}\in U\}$, which is actually the
same as $\{[\tau]_U\mid \boolval{\tau\in\check V}=\one\}$.
An easy induction on formulas shows that the map
$j:x\mapsto [\check x]_U$ is an elementary embedding
$j:V\to \Vbar$, and this is the map known as the Boolean
ultrapower. As was observed in
\cite{HamkinsSeabold:BooleanUltrapowers}, the critical
point of $j$ is the cardinality of the smallest maximal
antichain not met by the $U$, which in this case  must be
at least $\kappa$ since $U$ is $\kappa$-complete. If $\dot
G$ is the (usual) canonical name for the generic filter,
then $\boolval{\dot G\hbox{ is }\check V\hbox{-generic for
}\check\B}=1$, and so the corresponding equivalence class
$G=[\dot G]_U$ is $\Vbar$-generic for $[\check\B]_U=j(\B)$.
Since these embeddings and equivalence classes all exist in
$V$, we have the entire Boolean ultrapower
$$j:V\to\Vbar\of\Vbar[G]$$ existing in $V$, as desired. The
structure $\Vbar[G]$ is isomorphic to the quotient $V^\B/U$
by the map associating $[\tau]_U=[\dot\val(\check\tau,\dot
G)]_U$ in $V^\B/U$ with $\val([\check\tau]_U,G)$ in
$\Vbar[G]$.
\end{proof}

Certain instances of this phenomenon are already well
known. For example, consider Prikry forcing with respect to
a normal measure $\mu$ on a measurable cardinal $\kappa$,
which is $\ltkappa$-friendly because it adds no bounded
subsets to $\kappa$. If $V\to M_1\to M_2\to\cdots$ is the
usual iteration of $\mu$, with a direct limit to
$j_\omega:V\to M_\omega$, then the critical sequence
$\kappa_0,\kappa_1,\kappa_2,\ldots$ is well known to be
$M_\omega$-generic for the corresponding Prikry forcing at
$j_\omega(\kappa)$ using $j_\omega(\mu)$. This is precisely
the situation occurring in Theorem
\ref{Theorem.FriendlyStrCemb}, where we have an embedding
$j:V\to\Vbar$ and a $\Vbar$-generic filter $G\of j(\P)$ all
inside $V$. Thus, Theorem \ref{Theorem.FriendlyStrCemb}
generalizes this classical aspect about Prikry forcing to
all friendly forcing under the stronger assumption of
strong compactness.

We now derive Theorem
\ref{Theorem.InnerModelIndesctructibleSC} as a corollary.

\begin{proof}[Proof of Theorem \ref{Theorem.InnerModelIndesctructibleSC}]
We shall apply Theorem \ref{Theorem.FriendlyStrC} by
finding a $\ltkappa$-friendly version of the Laver
preparation. The original Laver preparation of
\cite{Laver78} is not friendly, because there are many
stages $\gamma<\kappa$ at which it definitely adds, for
example, a Cohen subset to $\gamma$. But a relatively
simple modification will make it $\ltkappa$-friendly.
Suppose that $\ell\from\kappa\to V_\kappa$ is a Laver
function. It follows easily that the restriction
$\ell\restrict(\gamma,\kappa)$ to any final segment
$(\gamma,\kappa)$ of $\kappa$ is also a Laver function, and
the corresponding Laver preparation
$\P_{\ell\restrict(\gamma,\kappa)}$ is $\leqgamma$-closed,
hence adding no new subsets to $\gamma$, while still
forcing indestructibility for $\kappa$. Let
$\P=\oplus\{\P_{\ell\restrict(\gamma,\kappa)}\mid
\gamma<\kappa\}$ be the lottery sum of all these various
preparations\footnote{If ${\mathcal A} = \{\P_i \mid i \in
I \}$, then the {\em lottery sum} $\oplus {\mathcal A}$ is
the partial order with underlying set $\{\langle \P, p
\rangle \mid \P \in {\mathcal A}\text{ and }p \in
\P\}\union\{\one\}$, ordered by $\langle \P, p \rangle \le
\langle \Q, q \rangle$ if and only if $\P = \Q$ and $p \le
q$, with $\one$ above everything. The lottery preparation
of \cite{Ham00} employs long iterations of such sums.},
%(that is, using side-by-side forcing),
so that the generic filter in effect selects a single
$\gamma$ and then forces with
$\P_{\ell\restrict(\gamma,\kappa)}$. This poset is
$\ltkappa$-friendly, since a condition could opt in the
lottery to use a preparation with $\gamma$ as large below
$\kappa$ as desired. The point is that the Laver
preparation works fine for indestructibility even if we
allow it to delay the start of the forcing as long as
desired, and such a modification makes it
$\ltkappa$-friendly. So Theorem \ref{Theorem.FriendlyStrC}
applies, and Theorem
\ref{Theorem.InnerModelIndesctructibleSC} now follows as a
corollary.
\end{proof}

After realizing that Theorem
\ref{Theorem.InnerModelIndesctructibleSC} could be proved
via Boolean ultrapowers, we searched for a direct proof. We
arrived at the following stronger result, which produces
more robust inner models $W$, satisfying a closure
condition $W^\theta\of W$.
%It turns out that these more
%robust inner models also arise from the Boolean ultrapower
%method, when it is employed with a supercompact cardinal.

\begin{theorem}\label{Theorem.IndestructibleSCrobust}
 If there is a supercompact cardinal, then for every cardinal $\theta$ there is an inner model $W$ with an indestructible supercompact cardinal, such that $W^\theta\of
 W$.
\end{theorem}

\begin{proof}
Suppose that $\kappa$ is supercompact. By a result of
Solovay \cite{Sol74}, the \SCH\ holds above $\kappa$, and
so if $\theta$ is any singular strong limit cardinal of
cofinality at least $\kappa$, then
$2^{\theta^\ltkappa}=\theta^\plus$. Consider any such
$\theta$ as large as desired above $\kappa$, and let
$j:V\to M$ be a $\theta$-supercompactness embedding, the
ultrapower by a normal fine measure on $P_\kappa(\theta)$.
Thus, $\theta<j(\kappa)$ and $M^\theta\of M$. By
elementarity, $j(\kappa)$ is supercompact in $M$. Let $\P$
be the Laver preparation of $j(\kappa)$ in $M$, with
nontrivial forcing only in the interval
$(\theta,j(\kappa))$. That is, we put off the start of the
Laver preparation until beyond $\theta$, and this is
exactly what corresponds to the use of friendliness in the
earlier proof. Notice that $\P$ is $\leqtheta$-closed in
$M$, and therefore also $\leqtheta$-closed in $V$. But
also, $\P$ has size $j(\kappa)$ in $M$, and has at most
$j(2^\kappa)$ many dense subsets in $M$. Observe in $V$
that $|j(2^\kappa)|\leq (2^\kappa)^{\theta^\ltkappa}\leq
(2^{\theta^\ltkappa})^{\theta^\ltkappa}=2^{\theta^\ltkappa}=\theta^\plus$.
In $V$ we may therefore enumerate the dense subsets of $\P$
in $M$ in a $\theta^\plus$ sequence, and using the fact
that $\P$ is $\leqtheta$-closed, diagonalize to meet them
all. So there is in $V$ an $M$-generic filter $G\of\P$.
Thus, $M[G]$ is an inner model of $V$, in which $j(\kappa)$
is an indestructible supercompact cardinal. Since
$M^\theta\of M$, it follows that $M[G]$ contains all
$\theta$-sequences of ordinals in $V$, and so also
$M[G]^\theta\of M[G]$. So $W=M[G]$ is as desired.
\end{proof}

This second method of proof can be generalized to the
following, where we define that $\P$ is {\df
$\ltkappa$-superfriendly}, if for every $\gamma<\kappa$
there is a condition $p\in\P$ such that $\P\restrict p$ is
$\leqgamma$-strategically closed. It was the
superfriendliness of the Laver preparation that figured in
the proof of Theorem \ref{Theorem.IndestructibleSCrobust}
and the proof generalizes in a straightforward way to
obtain the theorem below.

\begin{theorem}\label{Theorem.FriendlySC}
If $\kappa$ is supercompact and $\P$ is
$\ltkappa$-superfriendly, then for every $\theta$ there is
an inner model $W$ satisfying every statement forced by
$\P$ over $V$ and for which $W^\theta\of W$.
\end{theorem}
We can now solve several more of the test questions as
corollaries.

\begin{theorem}\label{Theorem.InnerModelSC+2^kappa=kappa+}
If there is a supercompact cardinal, then there is an inner
model with an indestructible supercompact cardinal $\kappa$
such that $2^\kappa=\kappa^\plus$, and another inner model
with an indestructible supercompact cardinal $\kappa$ such
that $2^\kappa=\kappa^\plusplus$. Thus, the answers to
Test Questions \ref{TestQuestion.Supercompact+2^kappa=kappa+},
\ref{TestQuestion.Supercompact+2^kappa>kappa+} and
\ref{TestQuestion.IndestructibleSC} are yes. Indeed, for
any cardinal $\theta$, such inner models $W$ can be found
for which also $W^\theta\of W$.
\end{theorem}

\begin{proof}
Let $\P$ be the $\ltkappa$-friendly version of the Laver
preparation used in Theorem
\ref{Theorem.InnerModelIndesctructibleSC}, which is easily
seen to be ${<}\kappa$-superfriendly, and let
$\Qdot=\dot\Add(\kappa^\plus,1)$ be the subsequent forcing
to ensure $2^\kappa=\kappa^\plus$. The combination
$\P*\Qdot$ remains $\ltkappa$-superfriendly, forces
$2^\kappa=\kappa^\plus$ and preserves the indestructible
supercompactness of $\kappa$. Thus, by either Theorem
\ref{Theorem.FriendlyStrC} or \ref{Theorem.FriendlySC},
there is an inner model satisfying this theory. Similarly,
if $\Rdot=\dot\Add(\kappa,\kappa^\plusplus)$, then
$\P*\Qdot*\Rdot$ is $\ltkappa$-superfriendly, preserves the
indestructible supercompactness of $\kappa$ and forces
$2^\kappa = \kappa^\plusplus$, so again there is an inner
model of the desired theory. The method of Theorem
\ref{Theorem.FriendlySC} will ensure in each case, for any
desired cardinal $\theta$, that the inner model $W$
satisfies $W^\theta\of W$.
\end{proof}

The proof admits myriad alternatives. For example, we could
have just as easily forced $2^\kappa=\kappa^\plusplusplus$,
or \GCH\ on a long block of cardinals at $\kappa$ and
above, or failures of this, in any definable pattern above
$\kappa$. If $\Q$ is any $\ltkappa$-directed closed
forcing, to be performed after the (superfriendly) Laver
preparation, then the combination $\P*\Qdot$ is
$\ltkappa$-superfriendly and preserves the indestructible
supercompactness of $\kappa$. Thus, any statement forced by
$\P*\Qdot$, using any parameter in $V_\kappa$, will be true
in the inner models $W$ arising in Theorems
\ref{Theorem.FriendlyStrC} and \ref{Theorem.FriendlySC}.
See also Theorem \ref{Theorem.HOD^W} for an application of
this method.

We now apply these methods to the family of questions
surrounding Test Questions \ref{TestQuestion.SC+V=HOD} and
\ref{TestQuestion.SC+HOD^W}. The following theorem answers
Test Question \ref{TestQuestion.SC+HOD^W} and several of its
variants, but not Test Question \ref{TestQuestion.SC+V=HOD}.
With our third proof method in a later section, we will
deduce the full conclusion of Test Question
\ref{TestQuestion.SC+V=HOD} using a slightly stronger
hypothesis.

\begin{theorem} \ \label{Theorem.HOD^W}
 \begin{enumerate}
 \item If $\kappa$ is strongly compact, then there is
     an inner model $W$ with a strongly compact
     cardinal, such that $H_{\kappa^\plus}^V\of
     \HOD^W$. If the \GCH\ holds below $\kappa$, then
     for any $A\in H_{\kappa^\plus}^V$, one can arrange
     that $A$ is definable in $W$ without parameters.
 \item If $\kappa$ is measurable and
     $2^\kappa=\kappa^\plus$, then there is an inner
     model $W$ with a measurable cardinal, such that
     $H_{\kappa^\plus}^V\of\HOD^W$. If the \GCH\ holds
     below $\kappa$, then for any $A\in
     H_{\kappa^\plus}^V$, one can arrange that $A$ is
     definable in $W$ without parameters.
 \item If $\kappa$ is supercompact, then for every
     cardinal $\theta$ and every set $A\in
     H_{\theta^\plus}^V$, there is an inner model $W$
     with a supercompact cardinal in which $A\in
     \HOD^W$ and $W^\theta\of W$. If the \GCH\ holds
     below $\kappa$, then one can arrange that $A$ is
     definable in $W$ without parameters.
     %If $\kappa$ is supercompact, then for every set
     %$A\in H_{\kappa^\plus}^V$ and every cardinal $\theta$, there is an inner
     %model $W$ with a supercompact cardinal in which
     %$A\in \HOD^W$ and $W^\theta\of W$. If the \GCH\
     %holds below $\kappa$, then one can arrange that $A$ is
     %definable in $W$ without parameters.
 \end{enumerate}
\end{theorem}

\noindent In particular, by Theorem \ref{Theorem.HOD^W}(3),
the answer to Test Question \ref{TestQuestion.SC+HOD^W} is yes.

\begin{proof}
For Statement (1), we use the Boolean ultrapower method of
Theorems \ref{Theorem.FriendlyStrC} and
\ref{Theorem.FriendlyStrCemb}. For $\gamma<\kappa$, let
$\Q_\gamma$ be the poset that codes $P(\gamma)$ into the
\GCH\ pattern on a block of cardinals above $\gamma$, using
$\leqgamma$-closed forcing. Let $\P$ be the lottery sum
$\oplus\{ \Q_\gamma \st \gamma<\kappa \}$, which is
$\ltkappa$-superfriendly and therefore $\ltkappa$-friendly.
Since each $\Q_\gamma$ is small relative to $\kappa$, it
follows by the results of \cite{LS67} that $\kappa$ remains
strongly compact after forcing with $\P$. By Theorem
\ref{Theorem.FriendlyStrCemb}, there is an embedding
$j:V\to\Vbar$ into an inner model $\Vbar$ with
$\kappa\leq\cp(j)$, and in $V$, there is a $\Vbar$-generic
filter $G\of j(\P)$. In particular, $V_\kappa =
\Vbar_\kappa$. Consequently, the ordinal $\gamma$ selected
by $G$ in the lottery $j(\P)$ must be at least $\kappa$.
Thus, $G$ is coding $P(\gamma)$ and hence also $P(\kappa)$
into the continuum function in some interval between
$\kappa$ and $j(\kappa)$. So in the inner model
$W=\Vbar[G]$, we have a strongly compact cardinal
$j(\kappa)$, as well as $P(\kappa)^V\of \HOD^W$ and hence
$H_{\kappa^\plus}^V\of\HOD^W$, as desired.

For the second part of Statement (1), in the case that the
\GCH\ holds below $\kappa$ in $V$, consider any $A\in
H_{\kappa^\plus}$. Let $\hat A\of\kappa$ be a subset of
$\kappa$ coding $A$ in some canonical way. Let $\P$ be the
forcing as in the previous paragraph, but modified so that
the lottery also may choose the order in which the sets in
$P(\gamma)$ are coded.
%By inspecting the proof of Theorem
%\ref{Theorem.FriendlyStrCemb}, i.e., by
%forcing below the appropriate condition,
%we may assume that $G$ opts
%for a poset in $j(\P)$ that begins coding at $\gamma$,
%which is the successor of a cardinal of cofinality
%$\kappa$, and that $\hat A$ is the first set to be coded.
First, we argue that we may assume that $G$ opts for a
poset in $j(\P)$ that begins coding at $\gamma$, which is
the successor of a cardinal of cofinality $\kappa$, and
that $\hat A$ is the first set to be coded. Note that the
statement $\varphi$ that $G$ makes such a choice is
expressible as an assertion in the forcing language, and
so, following the proof of Theorem
\ref{Theorem.FriendlyStrC}, it suffices to obtain an
ultrafilter $U$ for the complete Boolean algebra $\B$
corresponding to the poset $\P$ containing the Boolean
value of $\varphi$. Fixing a strong compactness embedding
$h$, we obtain $U$ precisely as in the proof of Theorem
\ref{Theorem.FriendlyStrC}, only making sure that the
condition $p\in h(\B)$ chosen to witness the friendliness
of $h(\B)$ for $\lambda=|s|^M<j(\kappa)$, where $j\image
\B\of s$, forces the statement $\varphi$ with $h$ applied
to the parameters. Now observe that in $W=\Vbar[G]$, the
cardinal $\gamma$ is definable as the cardinal up to which
the \GCH\ holds, since in $\Vbar$ the \GCH\ holds below
$j(\kappa)$ by elementarity. It follows that $\kappa$ is
definable as the cofinality of the predecessor of $\gamma$,
and so $\hat A$ and hence $A$ are definable in $W$ without
parameters. As in the previous paragraph, we also have a
strongly compact cardinal in $W$ and $H_{\theta^\plus}^V\of
W$.

For Statement (2), suppose that $\kappa$ is measurable and
$2^\kappa=\kappa^\plus$. We follow the method of Theorem
\ref{Theorem.IndestructibleSCrobust}. Let $j:V\to M$ be the
ultrapower by any normal measure on $\kappa$. Let $\P$ be
the forcing used to prove the second part of Statement (1),
which by lottery selects some $\gamma<\kappa$ and an
enumeration of $P(\gamma)$, which is then coded into the
\GCH\ pattern above $\gamma$. In the forcing $j(\P)$,
consider a condition $p$ that opts to code $P(\kappa)$.
Thus, $j(\P)\restrict p$ is $\leqkappa$-closed and has size
less than $j(\kappa)$. Since $2^\kappa=\kappa^\plus$, the
number of subsets of $j(\P)$ in $M$, counted in $V$, is
bounded by $|j(2^\kappa)|^V\leq
(2^\kappa)^\kappa=\kappa^\plus$. Thus, by diagonalization,
we may construct in $V$ an $M$-generic filter $G\of j(\P)$
below $p$. Let $W=M[G]$. By the results of \cite{LS67}, the
cardinal $j(\kappa)$ remains measurable in $W$, since below
$p$ the forcing was small relative to $j(\kappa)$. In
addition, every set in $P(\kappa)^V=P(\kappa)^M$ is coded
into the continuum function of $W$, so
$H_{\kappa^\plus}^V\of\HOD^W$, as desired. If the \GCH\
holds below $\kappa$, then it holds below $j(\kappa)$ in
$M$, and so we can define $\kappa$ in $W$ as the cardinal
up to which the \GCH\ holds, and hence define the first set
that is coded without parameters.

For Statement (3), where $\kappa$ is supercompact, we may
use the same argument as in Statements (1) and (2), but
employing the method of Theorem
\ref{Theorem.IndestructibleSCrobust}. Let $j:V\to M$ be a
$\theta$-supercompactness embedding, so that in particular
$M^\theta\of M$ and $\theta<j(\kappa)$. Let $\P$ be the
forcing from Statement (2), and in the forcing $j(\P)$,
consider a condition $p$ that opts to code $P(\theta)$.
Note that $j(\P)\restrict p$ is $\leq\theta$-closed. By the
proof of Theorem \ref{Theorem.IndestructibleSCrobust}, we
know that $V$ has an $M$-generic filter $G$ containing $p$.
Let $W=M[G]$, and note that $j(\kappa)$ remains
supercompact in $W$ by \cite{LS67}. Since $M^\theta\of M$,
it follows that $W^\theta\of W$ and
$H_{\theta^\plus}^V=H_{\theta^\plus}^M\of\HOD^W$, as
desired. If the \GCH\ holds below $\kappa$, then it holds
below $j(\kappa)$ in $M$, and so we can define $\theta$ in
$W$ as the cardinal up to which the \GCH\ holds and hence
define the first set that is coded without parameters.
\begin{comment}
For Statement (3), where $\kappa$ is supercompact, we may
use the same argument as in Statements (1) and (2), but
employing the method of Theorem
\ref{Theorem.IndestructibleSCrobust}. The resulting inner
model $W=M[G]$ will have $H_{\kappa^\plus}^V\of\HOD^W$ and
$W^\theta\of W$, as desired. And again, when the \GCH\
holds below $\kappa$ in $V$, then we may arrange that any
particular $A\in H_{\kappa^\plus}^V$ is definable in $W$
without parameters.
\end{comment}
\end{proof}

Let us highlight the consequences of this theorem with a
quick example. Namely, suppose that $\kappa$ is strongly
compact in $V$ and the \GCH\ holds. Both of these
statements remain true in the forcing extension $V[c]$
obtained by adding a single $V$-generic Cohen real $c$.
Since this forcing is almost homogeneous, we know $c$ is
not in $\HOD^{V[c]}$.
%(See \cite{Reitz2006:Dissertation}
%and \cite{Reitz2007:TheGroundAxiom} for the definition of
%almost homogeneous and a discussion of this property of
%almost homogeneous forcing.)
Nevertheless, by Theorem \ref{Theorem.HOD^W}, there are
inner models $W\of V[c]$ such that $H^V_{\kappa^+}
\subseteq {\rm HOD}^W$, with a strongly compact cardinal,
in which $c$ is definable without parameters!

Our first two proof methods were able to answer several of
the test questions with the provably optimal hypothesis
and, moreover, while also producing inner models with some
nice features, such as $W^\theta\of W$ for any desired
$\theta$. Nevertheless, and perhaps as a consequence, these
methods seem unable to produce inner models in which the
full \GCH\ holds, say, if the \CH\ fails in $V$, because
the resulting inner models for those methods will agree
with $V$ up to $V_\kappa$ and beyond, where $\kappa$ is the
initial supercompact cardinal. Similarly, neither method
seems able to produce an inner model in which the \rm{PFA}
holds, since the only known forcing to attain this---a long
countable support iteration of proper forcing---adds Cohen
reals unboundedly often and is therefore highly
non-friendly. Furthermore, the methods seem not easily to
accommodate class forcing, and allow us only to put
particular sets $A$ into $\HOD^W$ for an inner model $W$,
without having $W$ fully satisfy $V=\HOD$. Therefore, these
methods seem unable to answer Test Questions
\ref{TestQuestion.Supercompact+GCH}, \ref{TestQuestion.PFA}
and \ref{TestQuestion.SC+V=HOD}. (With our third proof
method, we shall give partial answers to these questions
in Theorems \ref{Theorem.InnerModelPFA} and
\ref{Theorem.InnerModelSC+GCH}, by using a stronger
hypothesis.) Another unusual feature of our first two
methods, as used in Theorems
\ref{Theorem.InnerModelIndesctructibleSC},
\ref{Theorem.IndestructibleSCrobust} and
\ref{Theorem.InnerModelSC+2^kappa=kappa+}, is that it is
not the same supercompact cardinal $\kappa$ that is found
to be supercompact in the desired inner model. Rather, it
is in each case the ordinal $j(\kappa)$ that is found to be
supercompact (and indestructible or with fragments or
failures of the \GCH) in an inner model. A modified version
of Test Question \ref{TestQuestion.IndestructibleSC} could ask,
after all, whether every supercompact cardinal $\kappa$ is
itself indestructibly supercompact in an inner model. For
precisely this question, we don't know, but if $\kappa$ is
supercompact up to a weakly iterable cardinal above
$\kappa$, then the answer is yes by Theorem
\ref{Theorem.InnerModelSCiterable}. (See Section
\ref{sec:iterability} for the definition of weakly iterable
cardinal.)

So let us now turn to the third method of proof, which will
address these concerns, at the price of an additional large
cardinal hypothesis. We shall use this method to produce an
inner model with a supercompact cardinal and the full \GCH,
an inner model of the \PFA\ and an inner model where
$\kappa$ itself is indestructibly supercompact, among other
possibilities. The method is very similar to the methods
introduced and fruitfully applied by Sy Friedman
\cite{Friedman2006:InternalConsistencyAndIMH} and by Sy
Friedman and Natasha Dobrinen
\cite{DobrinenFriedman2008:InternalConsistencyGlobalCoStationarity},
\cite{DobrinenFriedman:ConsistencyStrengthOfTreePropertyAtDoubleSuccessorMeasurable},
where they construct class generic filters in $V$ over an
inner model $W$. Also Ralf Schindler, in a personal
communication with the third author, used a version of the
method to provide an answer to Test Question
\ref{TestQuestion.PFA}, observing that if there is a
supercompact cardinal with a measurable cardinal above it,
then there is an inner model of the \PFA.

\begin{theorem}[Schindler]\label{Theorem.SC+MeasurableAbove}
If there is a supercompact cardinal with a measurable
cardinal above it, then there is an inner model of the
PFA.
\end{theorem}

The basic idea is that if $\kappa$ is supercompact and
$\kappa<\delta$ for some measurable cardinal $\delta$, then
one finds a countable elementary substructure $X\elesub
V_\theta$, with $\delta\ll\theta$, whose Mostowski collapse
is a countable iterable structure with a supercompact
cardinal $\kappa_0$ below a measurable cardinal $\delta_0$.
By iterating the measurable cardinal $\delta_0$ of this
structure out of the universe, one arrives at a full inner
model $M$, and because $\kappa_0$ was below the critical
point of the iteration, which is $\delta_0$, it follows
that both $\kappa_0$ and even $P(\kappa_0)^M$ are countable
in $V$. Thus, by the usual diagonalization in $V$, there is
an $M$-generic filter $G$ for the Baumgartner \PFA\ forcing
(or whatever other forcing was desired), and so $M[G]$ is
the desired inner model. This method generalizes to {\it
any} forcing notion below a measurable cardinal.

In the subsequent sections of this article, we shall
elaborate on the details of this argument, while also
explaining how to reduce the hypothesis from a measurable
cardinal above the supercompact cardinal to merely a weakly
iterable cardinal above. The construction encounters a few
complications in the class-length forcing iterations, since
(unlike the argument above) these iterations will be
stretched to proper class size during the iteration, and so
one cannot quite so easily produce the desired $M$-generic
filter. Nevertheless, the new method remains fundamentally
similar to the argument we described in the previous
paragraph. Finally, we shall give several additional
applications of the method.

\section{Iterable Structures}\label{sec:iterability}

We now develop some basic facts about iterable structures,
which shall be sufficient to carry out the third proof
method. In particular, we shall review the fact that any
structure elementarily embedding into an iterable structure
is itself iterable, and for a special class of forcing
required in later arguments, we shall give sufficient
conditions for a forcing extension of a countable iterable
structure to remain iterable.

Consider structures of the form  $\<M,\delta,U>$ where
$M\models\ZFC^{-}$ is transitive, $\delta$ is a cardinal in
$M$, and $U\subseteq\power(\delta)^M$. The set $U$ is an
\emph{$M$-ultrafilter}, if $\<M,\delta,U>\models`` U$ is a
normal ultrafilter''. An $M$-ultrafilter $U$ is
\emph{weakly anemable} if $U\cap A\in M$ for every set $A$
of size $\delta$ in $M$. By using only the equivalence
classes of functions in $M$, an $M$-ultrafilter suffices
for the usual ultrapower construction. It is easy to see
that $U$ is weakly amenable exactly when $M$ and the
ultrapower of $M$ by $U$ have the same subsets of $\delta$,
that is the ultrapower embedding is $\delta$-powerset
preserving. In this case, it turns out that one can define
the iterated ultrapowers of $M$ by $U$ to any desired
ordinal length. We say that $\<M,\delta,U>$ is
\emph{iterable} if $U$ is a weakly amenable $M$-ultrafilter
and all of these resulting iterated ultrapowers are
well-founded.

\begin{definition}
A cardinal $\delta$ is \emph{weakly iterable} if there is
an iterable structure $\<M,\delta,U>$ containing $V_\delta$
as an element.
\end{definition}

It is easy to see that measurable cardinals are weakly
iterable. Ramsey cardinals also are weakly iterable, since
if $\delta$ is Ramsey, every $A\subseteq\delta$ is an
element of an iterable structure $\<M,\delta, U>$ (see
\cite{mitchell:ramsey}) and so there is an iterable
structure containing a subset of $\delta$ that Mostowski
collapses to $V_\delta$.  On the other hand, a weakly
iterable cardinal need not even be regular. For example,
every measurable cardinal remains weakly iterable after
Prikry forcing, because the ground model iterable
structures still exist. More generally, we claim that the
least weakly iterable cardinal must have cofinality
$\omega$. To see this, suppose that $\delta$ is a weakly
iterable cardinal of uncountable cofinality with the
iterable structure $\<M,\delta, U>$. We shall argue that
there is a smaller weakly iterable cardinal of cofinality
$\omega$. Choose $X_0\prec M$ for some countable $X_0$
containing $\delta$, and let $\gamma_0=\text{sup}(X_0\cap
\delta)$. Inductively define $X_{n+1}\prec M$ with
$\gamma_{n+1}=\text{sup}(X_{n+1}\cap \delta)<\delta$
satisfying $V_{\gamma_n+1}\subseteq X_{n+1}$ and
$|X_{n+1}|<\delta$. This is possible since $\delta$ is
inaccessible in $M$, so the witnesses we need to add to
$X_{n+1}$ below $\delta$ will be bounded below $\delta$,
even if $\delta$ may be singular in $V$. Observe that if
$X_\omega=\bigcup_{n\in\omega}X_n$ and $\<N,\gamma,W>$ is
the collapse of the structure $\<X_\omega,\delta, U\cap
X_\omega>$, then $\delta$ collapses to $\gamma=\sup_n
\gamma_n$ and so $V_\gamma\in N$. The iterability of
$\<N,\gamma,W>$ will follow from Lemma
\ref{le:substructure} below, completing the argument that
$\gamma$ is a weakly iterable cardinal of cofinality
$\omega$ below $\delta$.

If $\delta$ is weakly iterable with the iterable structure
$\<M,\delta,U>$, then $\delta$ is at least ineffable in $M$
and therefore, the existence of weakly iterable cardinals
carries at least this large cardinal strength (see
\cite{gitman:ramsey}). In fact, weakly iterable cardinals
cannot exist in $L$ (see \cite{gitman:welch}), but it
follows from \cite{welch:unfoldables} that they are weaker
than an $\omega_1$-Erd\H{o}s cardinal. Note that the
inaccessibility of $\delta$ in the domain of the iterable
structure witnessing its weak iterability implies that it
is a $\beth$-fixed point and $V_\delta\models$ ZFC, by the
absoluteness of satisfaction.

\begin{lemma}\label{le:substructure}
Suppose $\<M,\delta,U>$ is iterable. Suppose further that
$\< N,\gamma,W>$ is a structure for which there exists an
elementary embedding $\rho:N\to M$ in the language
$\{\in\}$ with $\rho(\gamma)=\delta$ and the additional
property that whenever $x\in N$ is such that $x \subseteq
W$, then $\rho(x)\subseteq U$. Then $\<N,\gamma,W>$ is
iterable as well.
\end{lemma}
\begin{proof}
This is a standard idea. We shall demonstrate the
iterability of $\<N,\gamma,W>$ by elementarily embedding
the iterated ultrapowers of $N$ by $W$ into the iterated
ultrapowers of $M$ by $U$. Let $ \{j_{\xi\gamma}:M_\xi\to
M_\gamma\mid\xi<\gamma\in\Ord\} $ be the directed system of
iterated ultrapowers of $M=M_0$ with the associated
sequence of ultrafilters $\{U_\xi\mid\xi\in\Ord\}$, where
$U_0=U$. Also, let $\{h_{\xi\gamma}:N_\xi\to
N_\gamma\mid\xi<\gamma<\alpha\}$ be the not necessarily
well-founded directed system of iterated ultrapowers of
$N=N_0$ with the associated sequence of ultrafilters
$\{W_\xi\mid\xi\in\Ord\}$, where $W_0=W$. Let $\{W^i_0:i\in
I\}$ be any enumeration of all subsets of $W_0$ that are
elements of $N_0$, and define $W^i_\xi=h_{0\xi}(W^i_0)$. By
induction on $\xi$, it is easy to see that
$W_\xi=\Union_{i\in I}W^i_\xi$. We shall show that the
following diagram commutes:
\begin{diagram}[height=18pt]
M_0  &\rTo^{j_{01}}& M_1&\rTo^{j_{12}}&M_2&\rTo^{j_{23}}
 &\ldots&\rTo^{j_{\xi\xi+1}}&M_{\xi+1}&\rTo^{j_{\xi+1\xi+2}}&\ldots\\
\uTo_{\rho_0}& & \uTo_{\rho_1}& & \uTo_{\rho_2}& &  &
&\uTo_{\rho_{\xi+1}}&
& &\\
N_0 & \rTo^{h_{01}} & N_1
&\rTo^{h_{12}}&N_2&\rTo{h_{23}}&\ldots&\rTo^{h_{\xi\xi+1}}&N_{\xi+1}&\rTo^{h_{\xi+1\xi+2}}&\ldots
\end{diagram}
where
\begin{itemize}
\item[(1)]
    $\rho_{\xi+1}([f]_{W_\xi})=[\rho_\xi(f)]_{U_\xi}$,
\item[(2)] if $\lambda$ is a limit ordinal and $t$ is a
    thread in the direct limit $N_\lambda$ with domain
    $[\beta, \lambda)$, then
    $\rho_\lambda(t)=j_{\beta\lambda}(\rho_\beta(t(\beta)))$,
    and
\item[(3)] $\rho_\xi(W_\xi^i)\subseteq U_\xi$.
\end{itemize}
We shall argue that the $\rho_\xi$ exist by induction on
$\xi$. Let $\rho_0=\rho$, and note that $\rho_0$ satisfies
condition (3) by hypothesis. Suppose inductively that
$\rho_\xi:N_\xi\to M_\xi$ is an elementary embedding
satisfying condition (3). Define $\rho_{\xi+1}$ as in
condition (1) above. Using that $\rho_\xi(W_\xi^i)\subseteq
U_\xi$ by the inductive assumption, and $W_\xi=\Union
W_\xi^i$, it follows, in particular, that whenever $A\in
W_\xi$, then $\rho_\xi(A)\in U_\xi$. It follows that
$\rho_{\xi+1}$ is a well-defined map and an elementary
embedding. The commutativity of the diagram is also clear.
It remains to verify that
$\rho_{\xi+1}(W_{\xi+1}^i)\subseteq U_{\xi+1}$. Recall that
\begin{displaymath}
W_{\xi+1}^i=h_{\xi\xi+1}(W_{\xi}^i)=[c_{W_\xi^i}]_{W_\xi}.
\end{displaymath}
Let $\rho_\xi(W_\xi^i)=v$. Then by the inductive
assumption, we have $v\subseteq U_\xi$. Thus,
\begin{displaymath}
\rho_{\xi+1}(W_{\xi+1}^i)=[c_v]_{U_\xi}=j_{\xi\xi+1}(v)\subseteq
U_{\xi+1}.
\end{displaymath}
The last relation follows since $v\of U_\xi$. This
completes the inductive step. The limit case also follows
easily.
\end{proof}

Note that if $\rho$ is an elementary embedding in the
language with the predicate for the ultrafilter, then the
additional hypothesis of Lemma \ref{le:substructure}
follows for free. This is how Lemma \ref{le:substructure}
will be used in most applications below.

In the next section, we shall build  inner models by
iterating out these countable iterable structures and
forcing over the limit model inside the universe, just as
we explained in the proof sketch for Theorem
\ref{Theorem.SC+MeasurableAbove}. In other arguments,
however, the desired forcing will be stretched to proper
class length, and so we shall proceed instead by first
forcing over the countable structure and then iterating the
extended structure. For these arguments, therefore, we need
to understand when a forcing extension of an iterable
countable structure remains iterable. For a certain general
class of forcing notions and embeddings, we shall show in
Theorem \ref{th:forcing} that indeed the lift of an
iterable embedding to a forcing extension remains iterable,
and what is more, lifting just the first step of the
iteration to the forcing extension can lead to a lift of
the entire iteration. In rather general circumstances,
therefore, the iteration of a lift is a lift of the
iteration.

This argument will rely on the following characterization
of when an ultrapower of a forcing extension is a lift of
the ultrapower of the ground model. Suppose that $M$ is a
transitive model of $\ZFC^{-}$, that $\P$ is a poset in $M$
and that $G\subseteq\P$ is $M$-generic. Suppose further
that $U$ is an $M$-ultrafilter on a cardinal $\delta$ in
$M$ and $U^*$ is an $M[G]$-ultrafilter extending $U$, both
with well-founded ultrapowers. Then the ultrapower by $U^*$
lifts the ultrapower by $U$ if and only if every
$f:\delta\to M$ in $M[G]$ is $U^*$-equivalent to some
$g:\delta\to M$ in $M$. For the forward direction, suppose
that the ultrapower $j:M[G]\to N^*$ by $U^*$ lifts the
ultrapower $j:M\to N$ by $U$ and $\tau_G=f:\delta\to M$ is
a function in $M[G]$. Note that $f:\delta\to A$ where
$A=\{a\in M\mid \exists p{\in} \P\exists \xi{\in}\delta\,
p\forces \tau(\check\xi)=\check{a}\}$ is an element of $M$
by replacement. Thus, $j(f)(\delta)\in j(A)\of N$ and so
$j(f)(\delta)=j(g)(\delta)$ for some $g\in M$, from which
it follows that $f$ is $U^*$-equivalent to $g$. For the
backward direction, note that there is an isomorphism
between $N$ and a transitive submodel of $N^*$ sending
$[f]_U$ to $[f]_{U^*}$.
%To see why this no-new-functions
%characterization is true, note that the ultrapower by $U^*$
%must have the form of a forcing extension by the image of
%$\P$, with the ground model consisting exactly of the
%equivalence classes of functions $f:\delta\to M$ in $M[G]$.
Applying this characterization, if we lift the first
embedding in the iteration, then the ultrafilter derived
from the lift will have the above property. The key to the
argument will be to capture this property as a schema of
first-order statements over the forcing extension and
propagate it along the iteration using elementarity.

Let us now discuss a class of posets for which this
strategy proves successful. Suppose $j:M\to N$ is an
elementary embedding with critical point $\delta$. We
define that a poset $\P\in M$ is $j$-\emph{useful} if $\P$
is $\delta$-c.c.~in $M$ and $j(\P)\cong\P*\dot{\P}_\tail$,
where $\one_\P\forces ``\dot{\P}_\tail$ is
${\leq}\delta$-strategically closed'' in $N$. There are
numerous examples of such posets arising in the context of
forcing with large cardinals, and we shall mention several
in Sections \ref{Section.ThirdMethod} and
\ref{Section.FurtherApplications}. We presently explain how
the property of $j$-usefulness allows us to find lifts of
an ultrapower embedding to the forcing extension, so that
the iteration of the lift is the lift of the iteration. If
$\Q$ is any poset and $X$ is a set, not necessarily
transitive, define as usual that a condition $q \in \Q$ is
{\df $X$-generic for $\Q$} if for every $V$-generic filter
$G\of\Q$ containing $q$ and every maximal antichain
$A\of\Q$ with $A\in X$, the intersection $G\cap A\cap
X\neq\emptyset$; in other words, $q$ forces over $V$ that
the generic filter meets the maximal antichains of $X$
inside $X$. Suppose $j:M\to N$ and $\P$ is $j$-useful. Our
key observation about $j$-usefulness is that if $X\in N$ is
sufficiently elementary in $N$ with $X^{<\delta}\subseteq
X$ and $|X|=\delta$ in $N$, then every condition
$(p,\dot{q})\in \P*\dot{\P}_\tail\cap X$ can be
strengthened to an $X$-generic condition. First, observe
that every condition in $\P$ is $X$-generic for $\P$, since
maximal antichains of $\P$ have size less than $\delta$ and
so if $X$ contains such an antichain as an element, it must
be a subset as well. Thus, for the pair $(p,\dot{q})$ to be
$X$-generic for $j(\P)$, it suffices for $p$ to force that
$\dot{q}$ sits below some element of every dense subset of
$\P_\tail$ in $X[\dot{G}]$. Such a $\dot{q}$ is found by a
simple diagonalization argument, using the facts that
$|X|=\delta$, $X^{<\delta}\subseteq X$ and $\dot\P_\tail$
is forced to be $\leqdelta$-strategically closed.

Let us use the notation $\<M,\delta, U>\models``$I am
$\her{\delta}"$ to mean that $M$ believes every set has
size at most $\delta$. We now prove that if a certain
external genericity condition is met, then the iteration of
a lift is a lift of the iteration.

\begin{theorem}\label{th:forcing}
Suppose that $\<M_0,\delta,U_0>\models$``I am
$H_{\delta^\plus}$'' is iterable, and that the first step
of the iteration $j_{01}:M_0\to M_1$ lifts to an embedding
$j_{01}^*:M_0[G_0]\to M_1[G_1]$ on the forcing extension,
where $G_0\of\P$ is $M_0$-generic, $j_{01}^*(G_0)=G_1\of
j_{01}(\P)$ is $M_1$-generic and $\P$ is $j_{01}$-useful.
$$\small\begin{diagram}[height=15pt]
M_0[G_0]&\rTo^{\!\scriptscriptstyle j_{01}^*}& M_1[G_1]&\\
\cup &&\cup \\
M_0  &\rTo^{j_{01}}& M_1&\rTo^{j_{12}}&\cdots&\rTo&M_\xi
  &\rTo^{\!\!\scriptscriptstyle j_{\xi\xi+1}}&\cdots\\
\end{diagram}$$
Then $j_{01}^*$ is the ultrapower by a weakly amenable
$M_0[G_0]$-ultrafilter $U_0^*$ extending $U_0$.
Furthermore, if $G_1$ meets certain external dense sets
$D_a\of j_{01}(\P)$ for $a\in M_0$ described in the proof
below, then $\<M_0[G_0],\delta,U_0^*>$ is iterable, and the
entire iteration of $\<M_0[G_0],\delta,U_0^*>$ lifts the
iteration of $\<M_0,\delta,U_0>$ step-by-step.
$$\small\begin{diagram}[height=15pt]
M_0[G_0]&\rTo^{\scriptscriptstyle j_{01}^*}& M_1[G_1]&\rTo^{j_{12}^*}&\cdots&\rTo&M_\xi[G_\xi]&\rTo^{\scriptscriptstyle j_{\xi\xi+1}^*}&\cdots\\
\cup &&\cup &&&&\cup &&   \\
M_0  &\rTo^{j_{01}}& M_1&\rTo^{j_{12}}&\cdots&\rTo&M_\xi&\rTo^{\!\!\scriptscriptstyle j_{\xi\xi+1}}&\cdots\\\end{diagram}$$
\smallskip
Thus, the iteration of the lift is a lift of the iteration.
\end{theorem}

\begin{proof} Suppose that the ultrapower
$j_{01}:M_0\to M_1$ by $U_0$ lifts to $j_{01}^*:M_0[G_0]\to
M_1[G_1]$, with $j_{01}^*(G_0)=G_1$. By the normality of
$U_0$, it follows that every element of $M_1$ has the form
$j_{01}(f)(\delta)$ for some $f\in M_0^\delta\intersect
M_0$. Every element of $M_1[G_1]$ is $\tau_{G_1}$ for some
$j_{01}(\P)$-name $\tau\in M_1$, and so
$\tau=j_{01}(t)(\delta)$ for some function $t\in M_0$.
Define a function $f$ in $M_0[G_0]$ by
$f(\alpha)=t(\alpha)_{G_0}$, and observe that
$j_{01}^*(f)(\delta)=j_{01}(t)(\delta)_{j_{01}^*(G_0)}=\tau_{G_1}$.
Thus, every element of $M_1[G_1]$ has the form
$j_{01}^*(f)(\delta)$ for some $f\in
M_0[G_0]^\delta\intersect M_0[G_0]$. It follows that
$j_{01}^*$ is the ultrapower of $M_0[G_0]$ by the
$M[G_0]$-ultrafilter $U_0^*=\set{X\of \delta \mid X\in
M_0[G], \delta\in j_{01}^*(X)}$, which extends $U_0$. Note
that since $\P$ is $j_{01}$-useful, it follows that
$j_{01}(\P)\cong\P*\Ptail$, where $\Ptail$ adds no new
subsets of $\delta$ and $\P$ is $\delta$-c.c. From this, we
obtain that
$P(\delta)^{M_1[G_1]}=P(\delta)^{M_1[G_0]}=P(\delta)^{M_0[G_0]}$,
and so $U_0^*$ is weakly amenable to $M_0[G_0]$. It
therefore makes sense to speak of the iterated ultrapowers
of $\<M_0[G_0],\delta,U_0^*>$, apart from the question of
whether these iterates are well-founded.

The fact that the ultrapower $j_{01}:M_0\to M_1$ by $U_0$
lifts to the ultrapower $j_{01}^*:M_0[G_0]\to
M_1[j_{01}^*(G_0)]$ by $U_0^*$ is exactly equivalent to the
assertion that for every function $f\in
M_0^\delta\intersect M_0[G_0]$ there is a function $g\in
M_0^\delta\intersect M_0$ such that $f$ and $g$ agree on a
set in $U_0^*$. In slogan form: Every new function agrees
with an old function. This property is first order
expressible in the expanded structure
$\<M_0[G_0],\delta,U_0^*,M>$, by a statement with
complexity at most $\Pi_2$. If $j_{01}^*$ were sufficiently
elementary on this structure, then it would preserve the
truth of this statement and we could deduce easily that the
iterates of $U_0^*$ are step-by-step lifts of the
corresponding iterates of $U_0$, completing the proof.
Unfortunately, in the general case we cannot be sure that
$j_{01}^*$ is sufficiently elementary on this expanded
structure. Similarly, although the original embedding
$j_{01} : M_0 \to M_1$ is fully elementary, it may not be
fully elementary on the corresponding expanded structure
$j_{01}:\<M_0,\delta,U_0>\to \<M_1,\delta_1,U_1>$. The rest
of this argument, therefore, will be about getting around
this difficulty by showing that if $G_1$ satisfies an extra
genericity criterion, then the iteration of $U_0^*$ does
indeed lift the iteration of $U_0$.

Specifically, through this extra requirement on $G_1$, we
will arrange that for every $a\in M_0$, there is a set
$m_a\in M_0$ such that
\begin{itemize}\item[(1)] $m_a$ is a transitive model of $\ZFC^-$
    containing $\P$ and $a$, and
    % and $m_a^{<\delta}\subseteq m_a$ in $M_0[G_0]$. omitted this
\item[(2)] every $f:\delta\to m_a$ in $m_a[G_0]$ is
    $u_a^*$-equivalent to some $g:\delta\to m_a$ in
    $m_a$,
\end{itemize}
where $u_a^*=m_a\cap U_0^*$, which is an element of
$M_0[G_0]$ by the weak amenability of $U_0^*$ to
$M_0[G_0]$.

Let us first suppose that we have already attained (1) and
(2) for every $a$ and explain next how this leads to the
conclusion of the theorem. Suppose inductively that the
iteration of $U_0^*$ on $M_0[G_0]$ is a step-by-step lift
of the iteration of $U_0$ on $M_0$ up to stage $\xi$. Note
that limit stages come for free, because if every successor
stage before a limit is a lift, then the limit stage is
also a lift. Thus, we assume that the diagram in the
statement of the theorem is accurate through stage $\xi$,
so that in particular the $\xi^\th$ iteration
$j_{0\xi}^*:M_0[G_0]\to M_\xi[G_\xi]$ of $U_0^*$ is a lift
of the $\xi^\th$ iteration $j_{0\xi}:M_0\to M_\xi$ of
$U_0$, and we consider the next step $M_\xi[G_\xi]\to
\Ult(M_\xi[G_\xi],U_\xi^*)$. Since any given instance of
(1) and (2), for fixed $a$, is expressible in $M_0[G_0]$ as
a statement about $(m_a,G_0,u_a^*,a,\P)$, it follows by
elementarity that $j_{0\xi}^*(m_a)$ is a transitive model
of $\ZFC^{-}$ containing $j_{0\xi}(\P)$, and that every
$f:j_{0\xi}^*(\delta)\to j_{0\xi}^*(m_a)$ in
$j_{0\xi}^*(m_a)[G_\xi]$ is $j_{0\xi}^*(u_a^*)$-equivalent
to a function $g:j_{0\xi}^*(\delta)\to j_{0\xi}^*(m_a)$ in
$j_{0\xi}(m_a)$. Note that since $u_a^*\of U_0^*$, it
follows by an easy argument that $j_{0\xi}^*(u_a^*)\of
U_\xi^*$. Thus, as far as $j_{0\xi}^*(m_a)$ and
$j_{0\xi}^*(m_a)[G_\xi]$ are concerned, every new function
agrees with an old function. But now the key point is that
the $j_{0\xi}(m_a)$ exhaust $M_\xi$, since every object in
$M_\xi$ has the form $j_{0\xi}(f)(s)$ for some finite
$s\of\delta_\xi$, and thus once we put $f$ into $m_a$ by a
suitable choice of $a$, then $j_{0\xi}(f)(s)$ will be in
$j_{0\xi}(m_a)$. From this, it follows that the
$j_{0\xi}^*(m_a[G_0])$ exhaust $M_\xi[G_\xi]$, since every
element of $M_\xi[G_\xi]$ has a name in $M_\xi$. Therefore,
every new function in $M_\xi[G_\xi]$ agrees on a set in
$U_\xi^*$ with an old function in $M_\xi$, and so the
ultrapower of $M_\xi[G_\xi]$ by $U_\xi^*$ is a lift of
$j_{\xi\xi+1}$. Thus, we have continued the step-by-step
lifting one additional step, and so by induction, the
entire iteration lifts step-by-step as claimed.

It remains to explain how we achieve (1) and (2) for every
$a\in M_0$. First, we observe that $M_0$ is the union of
transitive models $m$ of $\ZFC$. This is because any set
$A\of\delta$ in $M_0$ is also in $M_1$ and therefore in
$V_{j_{01}(\delta)}^{M_1}$, which is a model of $\ZFC$
since $j_{01}(\delta)$ is inaccessible in $M_1$. By
collapsing an elementary substructure of this structure in
$M_1$, therefore, we find a size $\delta$ transitive model
$m\satisfies\ZFC$ with $A\in m\in M_1$. Since $m$ has size
$\delta$ and $M_0=H_{\delta^\plus}^{M_1}$ by weak
amenability, it follows that $m\in M_0$ as well. Thus, for
any $a\in M_0$ there are numerous models $m$ as in
Statement (1), even with full \ZFC.

For any such $m$, let $X_m = \set{j_{01}(f)(\delta) \mid
f\in m}$. It is not difficult to check that $X_m\elesub
j_{01}(m)$, by verifying the Tarski-Vaught criterion. Also,
since $j_{01}\restrict m \in M_1$, it follows that $X_m\in
M_1$, although by replacement the map $m\mapsto X_m$ cannot
exist in $M_1$, since $M_1$ is the union of all $X_m$. For
any $a\in M_0$, let
$$D_a=
\left\{\, q\in j_{01}(\P)\,\mid\,
  q\text{ is }X_m\text{-generic for some transitive }m\satisfies\ZFC^-\text{ with }
  a\in m\in M_0\,\right\}.$$
Recall that a condition $q$ is $X_m$-generic for
$j_{01}(\P)$ if every $M_1$-generic filter $G\of
j_{01}(\P)$ has $G\intersect D\intersect X_m\neq\emptyset$
for every dense set $D\of j_{01}(\P)$ in $M_1$. Because the
definition of $D_a$ refers to the various $X_m$, there is
little reason to expect that $D_a$ is a set in $M_1$.
Nevertheless, we shall argue anyway that it is a dense
subset of $j_{01}(\P)$.

To see this, fix $a$ and any condition $p\in j_{01}(\P)$.
Since $p=j(\vec p)(\delta)$ for some function $\vec p\in
M_0$, we may find as we explained above a transitive set
$m\in M_0$ with $a,\vec p,\P\in m\satisfies\ZFC$. We may
also ensure in that argument that $m^\ltdelta\of m$ in
$M_0$. It follows that $X_m^\ltdelta\of X_m$ in $M_1$, and
since $\vec p\in m$, we also know that $p=j(\vec
p)(\delta)\in X_m$. The forcing $j_{01}(\P)$ is in $X_m$
and factors as $\P*\Ptail$, where $\P$ is $\delta$-c.c. and
$\Ptail$ is forced to be $\leqdelta$-strategically closed.
Since $M_1$ knows that $X_m$ has size $\delta$, it can
perform a diagonalization below $p$ of the dense sets for
the tail forcing, and thereby produce a $\P$-name for a
condition in $\Ptail$ meeting all those dense sets. (This
is where we have used the key property of $j$-usefulness
mentioned before the theorem.) Thus, $M_1$ can build an
$X_m$-generic condition $q$ for $j_{01}(\P)$ below $p$.
This establishes that $D_a$ is dense, as we claimed.

We now suppose that $G_1$ meets all the dense sets $D_a$,
and use this to establish (1) and (2). For any $a\in M_0$,
we have a condition $q\in G_1$ that is $X_m$-generic for
some transitive $m\satisfies\ZFC^-$ in $M_0$ containing
$\<a,\P>$, thereby satisfying (1). From this, it follows
that $X_m[G_1]\intersect M_1 = X_m$, since for any name in
$X_m$ for an object in $M_1$, $X_m$ has a dense set of
conditions deciding its value, and since $G_1$ meets this
dense set inside $X_m$, the decided value must also be in
$X_m$. Now, suppose that $f:\delta\to m$ is a function in
$m[G_0]$, so that $f=\dot f_{G_0}$ for some name $\dot f\in
m$. Since $\dot f\in m$, it follows that $j_{01}(\dot f)\in
X_m$, and so $j_{01}(f)(\delta)\in X_m[G_1]$. Since
$\ran(f)\of m$, it follows that $\ran(j_{01}(f))\of
j_{01}(m)$, which is contained in $M_1$. Thus,
$j_{01}(f)(\delta)\in X_m[G_1]\intersect M_1$, which is
equal to $X_m$. But every element of $X_m$ has the form
$j_{01}(g)(\delta)$ for some function $g\in m$, and so
$j_{01}(f)(\delta)=j_{01}(g)(\delta)$ for such a function
$g$. It follows that $f$ and $g$ agree on a set in $U_0^*$
and we have established (2), completing the argument.
\end{proof}

A special case of the theorem occurs when $\P$ has size
smaller than $\delta$ in $M_0$. In this case, $\Ptail$ is
trivial and the extra genericity condition is automatically
satisfied, since the dense sets $D_a$ would be elements of
$M_1$. The nontrivial case of the theorem occurs when the
forcing $\P$ has size $\delta$, and its image is therefore
stretched on the ultrapower side. We are unsure about the
extent to which it could be true generally that the
iteration of a lift is a lift of the iteration. Surely some
hypotheses are needed on the forcing, since if $\P$ is an
iteration of length $\delta$ and $j(\P)$ adds new subsets
to $\delta$ at stage $\delta$, for example, then the lift
$j_{01}^*$ will not be weakly amenable, making it
impossible to iterate. Our $j$-usefulness hypothesis avoids
this issue, but we are not sure whether it is possible to
omit the external genericity assumption we made on $G_1$.
Nevertheless, this extra genericity assumption appears to
be no more difficult to attain in practice than ordinary
$M_1$-genericity. For example, in the case of countable
structures:

\begin{corollary}\label{Corollary.CountableIterationOfLift}
If $\<M,\delta,U>\models$``I am $H_{\delta^\plus}$'' is a
countable iterable structure and $\P\in M$ is useful for
the ultrapower of $M$ by $U$, then there is an $M$-generic
filter $G\of\P$ and $M[G]$-ultrafilter $U^*$ extending $U$
such that $\<M[G],\delta,U^*>$ is iterable, and the
iteration of $M[G]$ by $U^*$ is a step-by-step lift of the
iteration of $M$ by $U$.
\end{corollary}

\begin{proof}
This is simply a special case of the previous theorem. When
$M$ is countable, then there is no trouble in finding an
$M$-generic filter $G$ and $M_1$-generic filter $G_1$
satisfying the extra genericity requirement, since there
are altogether only countably many dense sets to meet.
\end{proof}

\section{The third proof method}\label{Section.ThirdMethod}
In this section, for the third proof method, we generalize
the proof sketch of Theorem
\ref{Theorem.SC+MeasurableAbove} given at the end of
Section \ref{sec:threeproofs}. For the arguments here, we
shall use the hypothesis of having a weakly iterable
cardinal $\delta$ with $V_\delta$ a model containing large
cardinals. We shall use the structure $\<M,\delta,U>$
witnessing the weak iterability of $\delta$ to produce a
countable iterable structure and build the inner model out
of the iterates of this structure or the iterates of its
forcing extension.
\begin{theorem}\label{Theorem.Iterable}
If $\<M,\delta,U>$ is iterable with a poset $\P\in
V_\delta^M$, then there is an inner model satisfying every
sentence forced by $\P$ over $V_\delta^M$.
\end{theorem}

\begin{proof} Let $\<M_0,\delta_0,U_0>$ be obtained by collapsing a countable
elementary substructure of $\<M,\delta,U>$ containing $\P$.
By Lemma \ref{le:substructure}, $\<M_0,\delta_0,U_0>$ is
iterable. Also, if $\Q$ is the collapse of the poset $\P$,
then by elementarity $\Q$ forces the same sentences over
$V_{\delta_0}^{M_0}$ that $\P$ forces over $V_\delta^M$.
Let $\{j_{\xi\eta}:M_\xi\to M_\eta \mid \xi < \eta \in
\Ord\}$ be the corresponding directed system of iterated
ultrapowers of $M_0$, and consider the inner model
$W=\bigcup_{\xi\in\Ord}j_{0\xi}(V^{M_0}_{\delta_0})$, which
is the cumulative part of the iteration lying below the
critical sequence. Since $V^{M_0}_{\delta_0}\prec W$ and
$V^{W}_{\delta_0}=V^{M_0}_{\delta_0}$, it follows that $\Q$
forces the same sentences over $V_{\delta_0}^{M_0}$ as over
$W$, and these are the same as forced by $\P$ over
$V_\delta^M$. Since $\Q$ lies below the critical point
$\delta_0$ of the iteration, the model $W$ contains only
countably many dense subsets of $\Q$ and so we can build a
$W$-generic filter $G$ directly. Thus, the model $W[G]$, an
inner model of $V$, satisfies the requirement of the
theorem.
\end{proof}

Let us now apply this theorem to the case of an
indestructible supercompact cardinal.

\begin{theorem}\label{Theorem.InnerModelSCiterable}
If $\kappa$ is $\ltdelta$-supercompact for a weakly
iterable cardinal $\delta$ above $\kappa$, then there is an
inner model in which $\kappa$ is an indestructible
supercompact cardinal.
\end{theorem}

\begin{proof}
Suppose that $\kappa$ is $\ltdelta$-supercompact for a
weakly iterable cardinal $\delta$ above $\kappa$ and the
weak iterability of $\delta$ is witnessed by an iterable
structure $\<M,\delta,U>$, with $V_\delta\in M$. In
particular, $\kappa$ is $\ltdelta$-supercompact in $M$.
Note that the Laver preparation $\P$ of $\kappa$ is small
relative to $\delta$ in $M$. Thus, by Theorem
\ref{Theorem.Iterable}, there is an inner model $W_0$
satisfying the theory forced by $\P$ over $V_\delta$. The
forcing $\P$, of course, makes $\kappa$ indestructibly
supercompact in $V_\delta^\P$, and so the inner model $W_0$
has an indestructible supercompact cardinal $\kappa_0$.

In order to prove the full claim, we must find a $W$ in
which $\kappa$ itself is indestructibly supercompact. For
this, let us look more closely at how the inner model $W_0$
arises from the proof of Theorem \ref{Theorem.Iterable}.
Specifically, the indestructible supercompact cardinal
$\kappa_0$ of $W_0$ arises inside a countable iterable
structure $M_0$, obtained by a Mostowski collapse of a
countable structure containing $\kappa$, and $\kappa_0$ is
below the critical point $\delta_0$ of the iteration. Thus,
$\kappa_0$ is not moved by the iteration and is therefore a
countable ordinal in $V$, even though it is indestructibly
supercompact in $W_0$. Since in particular $\kappa_0$ is
measurable in $W_0$, we may consider the internal system of
embeddings obtained by iterating a normal measure on
$\kappa_0$ in $W_0$. The successive images of $\kappa_0$
lead to the critical sequence $\{\kappa_\alpha\st \alpha\in
\Ord\}$, which is a closed unbounded class of ordinals,
containing all cardinals of $V$. It follows that $\kappa$
itself appears on this critical sequence, as the
$\kappa^\th$ element $\kappa=\kappa_\kappa$. In particular,
if $j:W_0\to W_\kappa$ is the $\kappa^\th$ iteration of the
normal measure, then $j(\kappa_0)=\kappa$, and so by
elementarity, $W_\kappa$ is an inner model in which
$\kappa$ itself is an indestructible supercompact cardinal.
\end{proof}

It should be clear that once there is an inner model $W$
containing an indestructible supercompact cardinal, and
this cardinal is a mere countable ordinal in $V$, then in
fact it can be arranged that any desired cardinal of $V$ is
an indestructible supercompact cardinal in an inner model.
For example, this argument shows that if there is a
cardinal that is supercompact up to a weakly iterable
cardinal, then there are inner models $W$ in which
$\aleph_1^V$ is indestructibly supercompact, or
$\aleph_2^V$ or $\aleph_\omega^V$ is indestructibly
supercompact, and so on, as desired.

The method also provides an answer to Test Question
\ref{TestQuestion.PFA}.

\begin{theorem}\label{Theorem.InnerModelPFA}
 If $\kappa$ is $\ltdelta$-supercompact for a weakly iterable cardinal
 $\delta$ above $\kappa$, then there is an inner model of the \PFA.
\end{theorem}

\begin{proof}
Let  $\<M,\delta,U>$ be an iterable structure containg
$V_\delta$. Then $\kappa$ is supercompact in $V_\delta$,
and so the Baumgartner forcing $\P \in V_\delta$ forces the
\PFA\ over $V_\delta$. Thus, by Theorem
\ref{Theorem.Iterable}, there is an inner model of the
\PFA.
\end{proof}

Let us return to Test Question
\ref{TestQuestion.Supercompact+GCH}, where we aim to
produce an inner model with a supercompact cardinal and the
full \GCH. In Theorem
\ref{Theorem.InnerModelSC+2^kappa=kappa+}, we approached
this, by finding inner models with a supercompact cardinal
$\kappa$ such that $2^\kappa=\kappa^\plus$ or such that
$2^\kappa=\kappa^\plusplus$, and the proof generalized to
get various \GCH\ patterns at or above $\kappa$. The proofs
of those theorems, however, relied on the friendliness of
the iteration up to $\kappa$, and so seem unable to attain
the full \GCH. For example, if \CH\ fails in $V$, then
there can be no friendly forcing of the \GCH. The third
proof method, however, does work to produce such an inner
model. We cannot apply Theorem \ref{Theorem.Iterable}
directly to the case of the poset forcing the \GCH, since
it is a class forcing over $V_\delta$. Following the proof
of Theorem \ref{Theorem.Iterable}, we would need at the
last step to obtain a generic for a class forcing over the
inner model $W$, and there is no obvious reason to suppose
that such a $W$-generic can be constructed. Instead, using
Theorem \ref{th:forcing}, we shall follow the modified
strategy of forcing over the countable iterable structure
first and then iterating out to produce the inner model.
Note that if the \GCH\ fails in $V$, then for large
$\theta$ one cannot expect to find the \GCH\ in the robust
type of inner models $W$ for which $W^\theta\of W$, since
such a property would inject the \GCH\ violations from $V$
into $W$.

The following theorem generalizes Theorem
\ref{Theorem.Iterable} to the case of class forcing with
respect to $V_\delta$.

\begin{theorem}\label{th:forcingiterable}
If $\<M,\delta,U>$ is iterable and $\P\of V_\delta^M$ is a
poset in $M$ and useful for the ultrapower by $U$, then
there is an inner model satisfying every sentence forced by
$\P$ over $V_\delta^M$.
\end{theorem}

\begin{proof}
We may assume without loss of generality that
$\<M,\delta,U>\models ``$I am $\her{\delta}"$. (If not,
replace $M$ with $H_{\delta^\plus}^M$ and observe that the
structure $\<H_{\delta^\plus}^M,\delta, U>$ remains
iterable since it has all the same functions $f:\delta\to
H_{\delta^\plus}^M$ as $M$ and therefore its iterates are
substructures of the corresponding iterates of $\<M,\delta,
U>$.) As in Theorem \ref{Theorem.Iterable}, let
$\<M_0,\delta_0,U_0>$ be a countable iterable structure
obtained by collapsing a countable elementary substructure
of $\<M,\delta, U>$ containing $\P$, and let $\Q$ be the
image of $\P$ under the collapse. Since $M_0$ is countable,
there is by Corollary
\ref{Corollary.CountableIterationOfLift} an $M_0$-generic
filter $G_0\of\Q$ and an $M_0[G_0]$-ultrafilter $U_0^*$
extending $U_0$ such that $\<M_0[G_0],\delta_0,U_0^*>$ is
iterable, and such that the iteration of $M_0[G_0]$ by
$U_0^*$ is a step-by-step lifting of the iteration of $M_0$
by $U_0$. Note that
$V_{\delta_0}^{M_0[G_0]}=V_{\delta_0}^{M_0}[G_0]$ satisfies
the theory forced by $\P$ over $V_\delta^M$. Let
$\{j_{\xi\eta}:M_\xi[G_\xi]\to M_\eta[G_\eta]\mid
\xi<\eta\in\Ord\}$ be the directed system of iterated
ultrapowers of $M_0[G_0]$, and consider
$W=\bigcup_{\xi\in\Ord} j_{0\xi}(V_{\delta_0}^{M_0[G_0]})$.
Since the iteration of $U_0^*$ lifts the iteration of $U_0$
on $M_0$ step-by-step, it follows that $W=\bar W[H]$, where
$\bar W=\bigcup_{\xi\in\Ord}j_{0\xi}(V_{\delta_0}^{M_0})$
and $H$ is the $\bar W$-generic filter arising from
$\bigcup_\xi j_{0\xi}(G_0)$ for the class forcing obtained
by $\bigcup_\xi j_{0\xi}(\Q)$.  By elementarity, $W$
satisfies the same sentences that are forced to hold over
$V_{\delta_0}^{M_0}$ by $\Q$, and these are the same as
those forced to hold over $V_\delta^M$ by $\P$.
\end{proof}

We may now apply Theorem \ref{th:forcingiterable} to
provide answers to Test Questions
\ref{TestQuestion.Supercompact+GCH} and
\ref{TestQuestion.SC+V=HOD}, from a stronger hypothesis.

\begin{theorem}\label{Theorem.InnerModelSC+GCH}
 If $\kappa$ is $\ltdelta$-supercompact for a weakly iterable cardinal
 $\delta$, then there is an inner model in which $\kappa$
 is supercompact and the \GCH\ plus $V=\HOD$ hold.
\end{theorem}

\begin{proof} Let $\<M,\delta,U>$ be an iterable structure containing
$V_\delta$, and as before, assume without loss of
generality that $M\models``\text{I am }H_{\delta^+}"$.
Observe that the canonical class forcing of the \GCH\ is
definable over $V_\delta$ and useful for the ultrapower
embedding. Note that although $\delta$ may be singular in
$V$, it is Mahlo (and more) in $M$, and so the forcing is
$\delta$-c.c$.$ inside $M$. By Theorem
\ref{th:forcingiterable}, there is an inner model with a
supercompact cardinal and the \GCH. To obtain an inner
model where $\kappa$ itself is supercompact, simply follow
the second part of the proof of Theorem
\ref{Theorem.InnerModelSCiterable}. One can similarly
obtain an inner model satisfying $V=\HOD$ without the \GCH\
by coding sets into the continuum function, making
essentially the same argument. (See, e.g., the coding
method used in \cite[Theorem 11]{Reitz2006:Dissertation} or
\cite[Theorem 11]{Reitz2007:TheGroundAxiom}.) If
$\GCH+V=\HOD$ is desired, as in the statement of the
theorem, then one should use a coding method compatible
with the \GCH. For example, the $\Diamond_\gamma^*$ coding
method used in
\cite{Brooke-Taylor2009:LargeCardinalsAndWellOrderingOnTheUniverse},
in conjunction with the proof of \cite[Theorem
11]{Reitz2006:Dissertation} or \cite[Theorem
11]{Reitz2007:TheGroundAxiom}, forces $\GCH+V=\HOD$ while
preserving supercompactness, and has the desired closure
properties for this argument.
\end{proof}

The hypotheses of Theorems
\ref{Theorem.InnerModelSCiterable},
\ref{Theorem.InnerModelPFA} and
\ref{Theorem.InnerModelSC+GCH} can be improved slightly,
since it is not required that $\delta$ is weakly iterable,
but rather only that \bigskip

(*) $\kappa$ is $\ltdelta$-supercompact inside an iterable
structure $\<M,\delta,U>$ where $V_\delta^M$ exists.
\bigskip

\noindent It is irrelevant assuming (*) whether
$V_\delta^M$ is the true $V_\delta$, since the only use of
that in our argument was to ensure that $\kappa$ was
$\ltdelta$-supercompact in $M$.

Next, we improve the iteration method to find more robust
inner models, which not only satisfy the desired theory,
but which also agree with $V$ up to $\delta$. This sort of
additional feature cannot be attained by iterating a
countable model out of the universe, which is ultimately
how our earlier instances of the iteration method
proceeded.

Suppose as usual that $\<M,\delta,U>$ is a structure where
$M\models{\rm ZFC^-}$, $\delta$ is a cardinal in $M$, and
$U$ is a weakly amenable $M$-ultrafilter. Suppose further
that $V_\delta^M$ exists. As a shorthand, let us refer to
these structures as \emph{weakly amenable}. A weakly
amenable structure that is closed under
$\ltdelta$-sequences is automatically iterable. This is
because it will be correct about the countable completeness
of the ultrafilter, which suffices for iterability (see
\cite{kunen:ultrapowers}). Moreover, closure under
$\ltdelta$-sequences implies that $\delta$ is inaccessible
and hence $V^M_\delta=V_\delta$. Thus, if there exists a
weakly amenable structure with $M^{<\delta}\subseteq M$,
then $\delta$ is weakly iterable. The existence of these
structures, however, has a significantly larger consistency
strength than the existence of a weakly iterable cardinal
that is between Ramsey and measurable cardinals (see
\cite{gitman:ramsey}).

\begin{theorem}\label{Theorem.UsefulSuperfriendly}
Suppose $\<M,\delta,U>$ is weakly amenable with
$M^\ltdelta\of M$. Suppose that $\P\subseteq V_\delta$ is a
poset in $M$ such that for every $\gamma<\delta$, there is
a condition $p\in \P$ such that $\P\restrict p$ is
$\leqgamma$-strategically closed and useful for the
ultrapower of $M$ by $U$. Then there is an inner model $W$
of $V$ satisfying every sentence forced by $\P$ over
$V_\delta$ and with $V_\delta^W=V_\delta$.
\end{theorem}

\begin{proof}
As usual, without loss of generality, we assume that
$M\models``\text{I am }H_{\delta^+}"$. We may also assume
that $M$ has size $\delta$, since if necessary, we may
replace $M$ with an elementary substructure
$M^*=\Union_\omega M_n$, where each $M_n\prec M$ of size
$\delta$ is constructed in the ultrapower so that $M_n\cap
U\in M_{n+1}$, and observe that the structure $\<
M^*,\delta, U>$ remains iterable by Lemma
\ref{le:substructure}. The hypothesis on $\P$ is a
superfriendly version of usefulness. Consider the first two
steps of the iteration
$$\begin{diagram}
M=M_0 & \rTo^{j_{01}} & M_1 & \rTo^{j_{12}} & M_2 .\\
\end{diagram}$$
Our strategy will be to lift the {\it second} step of the
iteration. We shall produce in $V$ a lift
$j_{12}^*:M_1[G_1]\to M_2[G_2]$, where $G_1\of j_{01}(\P)$
is $M_1$-generic and $j^*_{12}(G_1)=G_2$ is $M_2$-generic
for $j_{02}(\P)$, while also satisfying the extra
genericity requirement of Theorem \ref{th:forcing}. By that
theorem, therefore, the lift will be iterable and the
desired inner model will be obtained by iterating it out of
the universe.

To begin, note that the structure $\<M_1,\delta_1,U_1>$
arising from the ultrapower of $\<M,\delta,U>$ is certainly
iterable, since it was obtained after one step of the
iterable structure $\<M,\delta,U>$. In addition, the
assumptions on $M_0$ ensure that $M_1^\ltdelta\of M_1$ and
$|M_1|=\delta$, and also that $M_2^\ltdelta\of M_2$ and
$|M_2|=\delta$. By the superfriendly assumption on $\P$,
and using elementarity, we may find a condition $p\in
j_{01}(\P)$ below which $j_{01}(\P)$ is
$\ltdelta$-strategically closed in $M_1$, and hence truly
$\ltdelta$-strategically closed. By definition of
usefulness, $\P$ has $\delta$-c.c.\ in $M_0$ and
$j_{01}(\P)$ factors in $M_1$ as $\P*\Ptail$ with $\Ptail$
forced to be ${\leq}\delta$-strategically closed. By
elementarity, it follows that $j_{01}(\P)$ has
$j_{01}(\delta)$-c.c.\ in $M_1$ and $j_{01}(\P)$ factors in
$M_2$ as $j_{01}(\P)*j_{01}(\Ptail)$ with $j_{01}(\Ptail)$
forced to be ${\leq}j_{01}(\delta)$-strategically closed,
and hence $j_{01}(\P)$ is useful for $j_{12}$. It follows
that below the condition $(p,\dot \one_{j_{01}(\Ptail)})$,
the poset $j_{02}(\P)$ is $\ltdelta$-strategically closed.
Since there are only $\delta$ many dense subsets of
$j_{02}(\P)$ in $M_2$ and $M_2^\ltdelta\of M_2$, we may
diagonalize to find an $M_2$-generic filter $G_2\of
j_{02}(\P)$ below $p$ in $V$. It follows that $G_1=
G_2\restrict j_{01}(\delta)$ is $M_1$-generic for
$j_{01}(\P)$, and we may lift the embedding $j_{12}$ to
$j_{12}^*:M_1[G_1]\to M_2[G_2]$. We may furthermore arrange
in the diagonalization that $G_2$ also meets all the
external dense sets $D_a$ arising in Theorem
\ref{th:forcing}, since there are only $\delta$ many such
additional sets, and they can simply be folded into the
diagonalization. Thus, by Theorem \ref{th:forcing}, the
lift $j_{12}^*$ is iterable. Let $\{j_{1\xi}^*:M_1[G_1]\to
M_\xi[G_\xi]\}$ be the corresponding iteration, and let
$W=\Union_\xi V_{j_{1\xi}(\delta_1)}^{M_\xi[G_\xi]}$ be the
resulting inner model. This is the union of an elementary
chain, and so $W$ is an elementary extension of $M_1[G_1]$,
which satisfies all sentences forced by $\P$ over
$V_\delta$ and includes $H_{\delta^\plus}^M$. In
particular, $V_\delta\of W$ and so $V_\delta^W=V_\delta$,
completing the proof.
\end{proof}
\begin{theorem}\label{Theorem.InnerModelSC+V=HOD}
If $\kappa$ is indestructibly $\ltdelta$-supercompact in a
weakly amenable $\<M,U,\delta>$ with $M^{<\delta}\subseteq
M$, then there is an inner model $W$ satisfying $V=\HOD$ in
which $\kappa$ is indestructibly supercompact and for which
$V_\delta^W=V_\delta$.
\end{theorem}

\begin{proof}
Let $\<M,\delta,U>$ be weakly amenable with
$M^{<\delta}\subseteq M$. It follows that $\kappa$ is
indestructibly supercompact in $V_\delta$. Let $\P$ be the
forcing notion that first generically chooses (via a
lottery sum) an ordinal $\gamma_0$ in the interval
$[\kappa,\delta)$, and then performs an Easton support
iteration of length $\delta$. $\P$ does nontrivial forcing
at regular cardinals $\gamma$ in the interval
$[\gamma_0,\delta)$, with forcing that either forces the
\GCH\ to hold at $\gamma$ or to fail at $\gamma$, using the
lottery sum $\oplus\{\Add(\gamma^\plus,1),
\Add(\gamma,\gamma^\plusplus)\}$. An easy density argument
(see the proof of \cite[Lemma
13.1]{Friedman2009:Dissertation}) shows that any particular
set of ordinals below $\delta$ added by this forcing will
be coded into the \GCH\ pattern below $\delta$, and so $\P$
forces $V=\HOD$ over $V_\delta$. By indestructibility, the
forcing $\P$ preserves the indestructible supercompactness
of $\kappa$. Furthermore, the forcing $\P$ is definable in
$V_\delta$, and the choice of $\gamma_0$ makes the forcing
as closed as desired below $\delta$, as well as useful for
the ultrapower of $M$ by $U$. Thus, the hypotheses of
Theorem \ref{Theorem.UsefulSuperfriendly} are satisfied. So
by that theorem, there is an inner model $W$ satisfying
$V=\HOD$ and having $V_\delta^W=V_\delta$. Since $\kappa$
is below $\delta$, the critical point of the iteration of
$M$ by $U$, it is not moved by that iteration, and so
$\kappa$ is indestructibly supercompact in $W$.
\end{proof}

Next, we consider a variant of one of the questions
mentioned after Test Question \ref{TestQuestion.SC+HOD^W},
asking the extent to which sets can be placed into the
$\HOD$ of an inner model.

\begin{theorem}\label{Theorem.StrongRamseyHOD^W}
If $\kappa$ is strongly Ramsey, then for any $A\in
H_{\kappa^\plus}$, there is an inner model $W$ containing
$A$ and satisfying $V=\HOD$. If the \GCH\ holds below
$\kappa$, then one can arrange that $A$ is definable in $W$
without parameters.
\end{theorem}

\begin{proof}
From our earlier discussion, we know that $\kappa$ is
strongly Ramsey if every $A\in H_{\kappa^\plus}$ can be
placed into a weakly amenable structure $\<M,\kappa,U>$
with $M^{<\kappa}\subseteq M$.

Starting with a weakly amenable $\<M,\kappa, U>$ with
$M^{<\kappa}\subseteq M$ and $A\in M$, we use the same
forcing as in the proof of Theorem
\ref{Theorem.InnerModelSC+V=HOD} and appeal to Theorem
\ref{Theorem.UsefulSuperfriendly} to obtain an inner model
$W$ satisfying $V=\HOD$ and having $A\in W$, as desired.

Lastly, if the \GCH\ holds below $\kappa$, then as in
Theorem \ref{Theorem.HOD^W}, we may arrange the coding to
begin with coding $A$, and thereby make $A$ definable in
$W$ without parameters.
\end{proof}

\begin{corollary}\label{Corollary.StrongRamseysHOD^W}
If there is a proper class of strongly Ramsey cardinals,
then every set $A$ is an element of some inner model $W$
satisfying $V=\HOD$.
\end{corollary}

\begin{proof}
Under this hypothesis, every set $A$ is in
$H_{\delta^\plus}$ for some strongly Ramsey cardinal
$\delta$, and so is in an inner model $W$ satisfying
$V=\HOD$ by Theorem \ref{Theorem.StrongRamseyHOD^W}.
\end{proof}

To summarize the situation with our test questions, we have
provided definite affirmative answers to Test Questions
\ref{TestQuestion.Supercompact+2^kappa=kappa+},
\ref{TestQuestion.Supercompact+2^kappa>kappa+},
\ref{TestQuestion.IndestructibleSC} and
\ref{TestQuestion.SC+HOD^W}, along with several variants,
but have only provided the affirmative conclusion of
Test Questions \ref{TestQuestion.Supercompact+GCH},
\ref{TestQuestion.PFA} and \ref{TestQuestion.SC+V=HOD} from
the (consistency-wise) stronger hypothesis that there is a
cardinal supercompact up to a weakly iterable cardinal (or
at least supercompact inside an iterable structure). We do
not know if this hypothesis can be weakened for these
results to merely a supercompact cardinal. Perhaps either
Woodin's new approach to building non-fine-structural inner
models of a supercompact cardinal, or Foreman's approach of
\cite{For09} for constructing inner models of very large
cardinals, will provide the answers to these questions.

\section{Further Applications}\label{Section.FurtherApplications}

We shall now describe how variants of our methods can be
used to obtain a further variety of inner models. First,
using the methods of Theorem \ref{Theorem.InnerModelSC+GCH}
and a stronger hypothesis, we can obtain:
\begin{theorem}
Suppose that $\delta$ is a weakly iterable cardinal and a
limit of cardinals that are $\ltdelta$-supercompact. Then:
 \begin{enumerate}
  \item There is an inner model with a proper class of
      supercompact cardinals, all Laver indestructible.
  \item There is an inner model with a proper class of
      supercompact cardinals, where the \GCH\ holds.
  \item There is an inner model with a proper class of
      supercompact cardinals, where $V=\HOD$ and the
      Ground Axiom hold.
  \end{enumerate}
\end{theorem}
Of course, there are numerous other possibilities, for any
of the usual forcing iterations; we mention only these
three as representative. In each case, the natural forcing
has the same closure properties needed to support the
argument of Theorem \ref{Theorem.InnerModelSC+GCH}. In the
case of Statement (2), for example, one uses the canonical
Easton support forcing of the \GCH, and in Statement (3),
one uses any of the usual iterations that force every set
to be coded into the \GCH\ pattern of the continuum
function, a state of affairs that implies both $V=\HOD$ and
the Ground Axiom, the assertion that the universe was not
obtained by set forcing over any inner model (see
\cite{Hamkins2005:TheGroundAxiom},
\cite{Reitz2006:Dissertation} and
\cite{Reitz2007:TheGroundAxiom}).

For the next application of our methods, we show that there
are inner models witnessing versions of classical results
of Magidor \cite{Mag76}.

\begin{theorem} \ \label{Magidor}

\begin{enumerate}

\item\label{mag1} If $\kappa$ is $\ltdelta$-strongly
    compact for a weakly iterable cardinal $\delta$
    above $\kappa$, then there is an inner model in
    which $\kappa$ is both the least strongly compact
    and the least measurable cardinal.

\item\label{mag2} If there is a strongly compact
    cardinal $\kappa$, then there is an inner model in
    which the least strongly compact cardinal has only
    boundedly many measurable cardinals below it.

\item\label{mag3} If there is a supercompact cardinal
    $\kappa$, then for every cardinal $\theta$, there
    is an inner model $W$ in which the least strongly
    compact cardinal is the least supercompact cardinal
    and for which $W^\theta \subseteq W$.

\item\label{mag4} If $\kappa$ is
    $\ltdelta$-supercompact for a weakly iterable
    cardinal $\delta$ above $\kappa$, then there is an
    inner model in which $\kappa$ is both the least
    strongly compact and least supercompact cardinal.

\end{enumerate}
\end{theorem}

\begin{proof}
For (\ref{mag1}), let ${\mathbb P}$ be Magidor's notion of
iterated Prikry forcing from \cite{Mag76}, which adds a
Prikry sequence to every measurable cardinal below
$\kappa$. Since $|\P|=\kappa$, it is small with respect to
$\delta$. By the arguments of \cite{Mag76}, the cardinal
$\kappa$ becomes both the least strongly compact and the
least measurable cardinal in $V^\P_\delta$. Thus, by
Theorem \ref{Theorem.Iterable}, there is an inner model in
which the least strongly compact cardinal is the least
measurable cardinal, and by the methods from the second
part of the proof of Theorem
\ref{Theorem.InnerModelSCiterable}, this cardinal may be
taken as $\kappa$ itself.

For (\ref{mag2}), we begin by noting that the partial
ordering ${\mathbb P}$ mentioned in the preceding paragraph
is not ${<}\kappa$-friendly. However, in analogy to the
first proof given for Theorem
\ref{Theorem.InnerModelIndesctructibleSC}, for every
$\gamma < \kappa$, let ${\mathbb P}_\gamma$ be Magidor's
notion of iterated Prikry forcing from \cite{Mag76} which
adds a Prikry sequence to every measurable cardinal in the
open interval $(\gamma, \kappa)$. By \cite{Mag76}, forcing
with ${\mathbb P}_\gamma$ adds no subsets to $\gamma$. Let
${\mathbb P}^* = \oplus \{{\mathbb P}_\gamma \mid \gamma <
\kappa\}$ be their lottery sum. Magidor's arguments of
\cite{Mag76} together with ${\mathbb P}^*$'s definition as
a lottery sum show that after forcing with ${\mathbb P}^*$,
for some $\gamma < \kappa$, $\kappa$ is the least strongly
compact cardinal, and there are no measurable cardinals in
the open interval $(\gamma, \kappa)$. Since ${\mathbb P}^*$
is ${<}\kappa$-friendly, (\ref{mag2}) now follows by
Theorem \ref{Theorem.FriendlyStrC}.

For (\ref{mag3}), let $\gamma < \kappa$, and let ${\mathbb
P}_\gamma$ be the Easton support iteration of length
$\kappa$ which adds a non-reflecting stationary set of
ordinals of cofinality $\omega$ to every non-measurable
regular limit of strong cardinals in the open interval
$(\gamma, \kappa)$. (In other words, ${\mathbb P}_\gamma$
does trivial forcing except at those $\delta \in (\gamma,
\kappa)$ which are non-measurable regular limits of strong
cardinals, where it adds a non-reflecting stationary set of
ordinals of cofinality $\omega$ to $\delta$.) By the
remarks in \cite[Section 2]{Apt05}, after forcing with
${\mathbb P}_\gamma$, $\kappa$ becomes both the least
strongly compact and least supercompact cardinal. Let
${\mathbb P} = \oplus \{{\mathbb P}_\gamma \mid \gamma <
\kappa\}$ be their lottery sum. Since ${\mathbb P}$ is
${<}\kappa$-superfriendly, (\ref{mag3}) now follows by
Theorem \ref{Theorem.FriendlySC}.

Finally, for (\ref{mag4}), we note that for any $\gamma <
\kappa$, ${\mathbb P}_\gamma$ of the preceding paragraph is
${<}\kappa$-friendly, since it is
${<}\kappa$-superfriendly. Because $|{\mathbb P}_\gamma| =
\kappa$, (\ref{mag4}) now follows by Theorem
\ref{Theorem.Iterable} and the methods from the second part
of the proof of Theorem \ref{Theorem.InnerModelSCiterable}.
\end{proof}

Before stating our next application, we briefly recall some
definitions. Say that a model $V$ of \ZFC${}$ containing
supercompact cardinals satisfies {\em level by level
equivalence between strong compactness and
supercompactness} if for every $\kappa < \lambda$ regular
cardinals, $\kappa$ is $\lambda$ strongly compact if and
only if $\kappa$ is $\lambda$ supercompact. Say that a
model $V$ of \ZFC${}$ containing supercompact cardinals
satisfies {\em level by level inequivalence between strong
compactness and supercompactness} if for every
non-supercompact measurable cardinal $\kappa$, there is
some $\lambda > \kappa$ such that $\kappa$ is $\lambda$
strongly compact yet $\kappa$ is not $\lambda$
supercompact. Models satisfying level by level equivalence
between strong compactness and supercompactness were first
constructed in \cite{ApSh97}, and models satisfying level
by level inequivalence between strong compactness and
supercompactness have been constructed in \cite{Apt02},
\cite{Apt10} and \cite{apter:lbli}.

\begin{theorem}\ \label{Apter}
\begin{enumerate}
\item\label{apt1} If the \GCH${}$ holds and $\kappa$ is
    $\ltdelta$-supercompact for a weakly iterable
    cardinal $\delta$ above $\kappa$, then there is an
    inner model in which $\kappa$ is supercompact and
    the \GCH${}$ and level by level equivalence between
    strong compactness and supercompactness hold.
\item\label{apt2} If the \GCH${}$ holds and $\kappa$ is
    $\delta$-supercompact for a weakly iterable
    cardinal $\delta$ above $\kappa$, then there is an
    inner model in which $\kappa$ is supercompact and
    the \GCH${}$ and level by level inequivalence
    between strong compactness and supercompactness
    hold.
\end{enumerate}
\end{theorem}

\begin{proof}
For (\ref{apt1}), assume that $\kappa$ and $\delta$ are
least such that $\kappa$ is $\ltdelta$-supercompact and
$\delta$ is a weakly iterable cardinal. Let $\P$ be the
class forcing from \cite{ApSh97} defined over $V_\delta$
such that $V^\P_\delta \satisfies ``\kappa$ is
supercompact, and the \GCH${}$ and level by level
equivalence between strong compactness and supercompactness
hold''. We refer readers to \cite[Section 3]{ApSh97} for
the exact definition of $\P$, which is rather complicated.
We do note, however, that if $\langle M, \delta, U \rangle$
is an iterable structure containing $V_\delta$, then
because $\kappa$ and $\delta$ are least such that $\kappa$
is ${<}\delta$-supercompact and $\delta$ is weakly
iterable, $\P$ is useful for the ultrapower embedding.
Therefore, by Theorem \ref{th:forcingiterable}, there is an
inner model containing a supercompact cardinal in which the
\GCH${}$ and level by level equivalence between strong
compactness and supercompactness hold, and by the methods
from the second part of the proof of Theorem
\ref{Theorem.InnerModelSCiterable}, this cardinal may be
taken as $\kappa$ itself.

For (\ref{apt2}), assume that $\kappa$ and $\delta$ are
least such that $\kappa$ is $\delta$-supercompact and
$\delta$ is a weakly iterable cardinal. It is a general
fact that if $\gamma$ is $\rho$-supercompact, $\rho$ is
weakly iterable, and $j : V \to M$ is an elementary
embedding witnessing the $\rho$-supercompactness of
$\gamma$, then $M \models ``\rho$ is weakly iterable''.
Therefore, by the proof of \cite[Theorem 2]{Apt10}, there
are cardinals $\kappa_0 < \delta_0 < \kappa$ and a partial
ordering $\P \in V_{\delta_0}$ such that $V^\P_{\delta_0}
\satisfies ``\kappa_0$ is supercompact, and the \GCH${}$
and level by level inequivalence between strong compactness
and supercompactness hold''. Thus, by Theorem
\ref{Theorem.Iterable}, there is an inner model containing
a supercompact cardinal in which the \GCH${}$ and level by
level inequivalence between strong compactness and
supercompactness hold, and by the methods from the second
part of the proof of Theorem
\ref{Theorem.InnerModelSCiterable}, this cardinal may be
taken as $\kappa$ itself.
\end{proof}

We mentioned at the opening of this article that we take
our test questions as representative of the many more
similar questions one could ask, inquiring about the
existence of inner models realizing various large cardinal
properties usually obtained by forcing. We would similarly
like to take our answers---and in particular, the three
proof methods we have described---as providing a key to
answering many of them. Indeed, we encourage the reader to
go ahead and formulate similar interesting questions and
see if these methods are able to provide an answer.
%(Perhaps there may be many good Masters Thesis topics in
%this vein.)
Going forward, we are especially keen to find or learn of
generalizations of our first two methods, in Theorems
\ref{Theorem.FriendlyStrC}, \ref{Theorem.FriendlyStrCemb}
and \ref{Theorem.FriendlySC}, which might allow us to find
the more robust inner models provided by these methods for
a greater variety of situations.

\bibliographystyle{alpha}
\bibliography{Innermodels}

\end{document}